\documentclass[12pt]{iopart}

\eqnobysec 

\usepackage{amsthm} 
\usepackage{amssymb} 
\usepackage{graphicx}

\def\R{\mathbb{R}}

\newtheorem{theorem}{Theorem}[section]
\newtheorem{proposition}[theorem]{Proposition}
\newtheorem{corollary}[theorem]{Corollary}
\newtheorem{lemma}[theorem]{Lemma}

\newtheorem{remark}[theorem]{Remark}

\def\cU{\mathcal{U}}
\def\cW{\mathcal{W}}
\def\cX{\mathcal{X}}
\def\cY{\mathcal{Y}}
\def\cI{\mathcal{I}}
\def\bp{\mathbf{p}}
\def\bu{\mathbf{u}}
\def\bv{\mathbf{v}}
\def\bw{\mathbf{w}}
\def\by{\mathbf{y}}
\def\bby{\mathbb{Y}}
\def\bz{\mathbf{z}}
\def\uy{\underline{y}}

\def\budag{\mathbf{u}^\dagger}
\def\thdag{\theta^\dagger}
\def\bF{\mathbb{F}}
\def\bR{\mathbb{R}}
\def\ddt{{\textstyle\frac{d}{dt}}}
\def\eps{\varepsilon}
\newcommand{\dup}[3]{\langle{#1},{#2}\rangle_{#3^*,#3}}
\newcommand{\eqref}[1]{(\ref{#1})}

\begin{document}
\title[All-at-once versus reduced iterative methods for time dependent inverse problems]{All-at-once versus reduced iterative methods for time dependent inverse problems}

\author{B Kaltenbacher}

\address{Alpen-Adria-Universit\"at Klagenfurt, Universit\"atstra{\ss}e 65-67, A-9020 Klagenfurt, Austria}

\ead{\mailto{Barbara.Kaltenbacher@aau.at}}
\begin{abstract}
In this paper we investigate all-at-once versus reduced regularization of dynamic inverse problems on finite time intervals $(0,T)$. In doing so, we concentrate on iterative methods and nonlinear problems, since they have already been shown to exhibit considerable differences in their reduced and all-at-once versions, whereas Tikhonov regularization is basically the same in both settings.
More precisely, we consider Landweber iteration, the iteratively regularized Gauss-Newton method, and the Landweber-Kaczmarz method, the latter relying on cyclic iteration over a subdivision of the problem into subsequent subintervals of $(0,T)$.
Part of the paper is devoted to providing an appropriate function space setting as well as establishing the required differentiability results needed for well-definedness and convergence of the methods under consideration. 
Based on this, we formulate and compare the above mentioned iterative methods in their all-at-once and their reduced version. Finally, we provide some convergence results in the framework of Hilbert space regularization theory and illustrate the convergence conditions by an example of an inverse source problem for a nonlinear diffusion equation. 
\end{abstract}

\maketitle

\section{Introduction}
\label{sec:Intro}

A large number of inverse problems in applications ranging from engineering via economics to systems biology can be formulated as a state space system
\begin{eqnarray}
\dot{\bu}(t)=f(t,\bu(t),\theta)+\bw^\delta(t) \quad t\in(0,T)\,, \quad\bu(0)=u_0(\theta)
\mbox{ and }\label{mod}\\
\by^\delta(t)=g(t,\bu(t),\theta)+\bz^\delta(t)\quad t\in(0,T) 
\label{obs}\\
\mbox{ or } \nonumber\\
\uy^\delta_i=g_i(\bu(t_i),\theta)+\bz_i^\delta \quad i\in\{1,\ldots,n\},
\label{obsdiscr}
\end{eqnarray}
where the dot denotes differentiation with respect to time and $\theta$ is a finite or infinite dimensional parameter that is supposed to be identified from the additional continuous or discrete observations \eqref{obs} or \eqref{obsdiscr}, respectively.
As relevant in applications, we also consider perturbations both in the model \eqref{mod} and in the observations \eqref{obs}, or \eqref{obsdiscr}, respectively. We will assume $\bw^\delta$ and $\bz^\delta$ to be unknown, only knowledge on bounds of their norms will be required to obtain convergence results. 
An exact solution $(\budag,\thdag)$ of the inverse problem is supposed to satisfy the unperturbed system
\begin{eqnarray}
\dot{\bu}^\dagger(t)=f(t,\budag(t),\thdag) \quad t\in(0,T)\,, \quad\bu(0)=u_0(\theta)
\mbox{ and }\label{modex}\\
\by(t)=g(t,\budag(t),\thdag)\quad t\in(0,T) 
\label{obsex}\\
\mbox{ or } \nonumber\\
\uy_i=g_i(\budag(t_i),\thdag) \quad i\in\{1,\ldots,n\},
\label{obsdiscrex}
\end{eqnarray}
and will --- as a prerequisite for its computational approximation, which will be in the focus of this paper --- be assumed to exist throughout this paper. (For sufficient conditions of well-posedness of the model equation for fixed parameter $\theta$ we refer, e.g., to \cite{Roubicek,Wloka} and the references therein.) However, unless otherwise stated, uniqueness of $(\budag,\thdag)$ will not be be presumed, here.

\medskip

In this paper we will mainly focus on continuous observations \eqref{obs}, \eqref{obsex}. The necessary modifications when dealing with discrete observations \eqref{obsdiscr}, \eqref{obsdiscrex} will be mentioned in separate remarks.

\medskip 

The aim of this paper is to formulate inverse problems of this kind in a reduced and an all-at-once fashion and compare some iterative regularization methods in these two different settings as they have already been shown to exhibit crucial differences for stationary inverse problems \cite{aao16}.
In particular, we will consider Landweber iteration \cite{HNS95} and the iteratively regularized Gauss-Newton method \cite{Baku92,BNS97}, as well as the Landweber-Kaczmarz iteration \cite{HKLS07,HLS07,KS02}, as the latter obviously lends itself to a subdivision of the problem to a union of time intervals $[0,T]=\bigcup_{j\in \{1,\ldots,m\}}[\tau_{j-1},\tau_j]$.

\medskip

Dynamic inverse problems and their numerical solution have recently attracted much interest, driven by applications ranging from dynamic imaging via time domain inverse scattering to the  identification of material properties from transient experiments, see, e.g., \cite{BurgerRossmanithZhang16,ChenHaddarLechleiterMonk10,Hahn14,KirschRieder16,LechleiterSchlasche17,SchusterWoestehoff14} to name just a few examples.
Recent work on general parameter identification and regularization approaches for dynamic models can, e.g., be found in \cite{ArnoldCalvettiSomersalo13,HaoNguyen12,KindermannLeitao07,Baumeister97,Kuegler10}
in the frameworks of statistic filtering, linear semigroup theory, dynamic programming, and online parameter identification, respectively. 
The present paper --- though dealing with the common topic of methods for solving dynamic inverse problems --- differs from these in the sense that we here consider a deterministic and nonlinear setting on a finite time horizon.

\bigskip   

We consider the model equation \eqref{mod} in the standard function space setting 
\[
\hspace*{-1cm}
\bu\in W^{1,p,p^*}(0,T;V,V^*)=\{\bv\in L_p(0,T;V)\, : \, \dot{\bu}\in L_{p^*}(0,T;V^*)\}
\subseteq C(0,T;H)\,,
\]
with $p^*=\frac{p}{p-1}$ and $V\subseteq H\subseteq V^*$ forming a Gelfand triple, i.e., $H$ a Hilbert space (to be identified with itself) with dense and continuous embedding $V\subseteq H$, consequently also $H\subseteq V^*$ cf., e.g., \cite{Roubicek,Wloka}. Additionally we will assume $H$  and $V$ to be separable.

The functions defining the model and observation equations 
\begin{eqnarray}
f:(0,T)\times V\times\cX\to V^*\,, \label{f}\\ 
g:(0,T)\times V\times\cX\to Z \mbox{ or } g_i:V\times\cX\to Z \label{g}
\end{eqnarray}
are, for fixed $\theta\in \cX$, assumed to be Caratheodory mappings, i.e., for all $(v,\theta)\in V\times\cX$ fixed, $f(\cdot,v,\theta)$, $g(\cdot,v,\theta)$ are measurable and for (almost) all $(t,\theta)\in(0,T)\times\cX$, $f(t,\cdot,\theta)$, $g(t,\cdot,\theta)$ are continuous. Moreover, they will be supposed to satisfy certain growth conditions to guarantee appropriate mapping properties of the induced Nemitskii operators on the Bochner space $L_p(0,T;V)$.
Moreover, the initial data is assumed to be contained in $H$
\begin{equation}\label{u0}
u_0:\cX\to H\,.
\end{equation}
Corresponding to this setting, the function spaces in which the perturbations live are 
\[
\bw^\delta\in L_{p'}(0,T;V^*), \quad \bz^\delta\in L_{q}(0,T;Z)\,.
\]

\medskip

In this paper we restrict ourselves to the Hilbert space framework, i.e., $p=p'=q=2$, $\cX,V,V^*,Z$ are Hilbert spaces and hence, so are 
\begin{equation}\label{UWY}
\cU=W^{1,2,2}(0,T;V,V^*)\,, \quad \cW=L_2(0,T;V^*)\,, \quad \cY=L_2(0,T;Z)\,,
\end{equation}
to avoid certain technicalities such as the use of nonlinear duality mappings and the somewhat more complicated stepsize choice for Landweber iteration in general Banach spaces in our exposition. 

\medskip

\begin{remark}\label{rem:timedependenttheta}
Note that the theory developed here in principle also comprises the case of time dependent parameters $\theta$ via the use of a space of time dependent functions as parameter space $\cX$. In particular, the case of a parameter evolving simultaneously with the state 
\begin{eqnarray*}
\dot{\bu}(t)=\tilde{f}(t,\bu(t),\theta(t))+\bw^\delta(t) \quad t\in(0,T)\,, \quad\bu(0)=u_0(\theta_0)
\\
\by^\delta(t)=\tilde{g}(t,\bu(t),\theta(t))+\bz^\delta(t)\quad t\in(0,T) 
\\
\mbox{ or } \nonumber\\
\uy^\delta_i=\tilde{g}_i(\bu(t_i),\theta(t_i))+\bz_i^\delta \quad i\in\{1,\ldots,n\}
\end{eqnarray*}
with 
\begin{eqnarray*}
\tilde{f}:(0,T)\times V\times X\to V^*\,, \\ 
\tilde{g}:(0,T)\times V\times X\to Z \mbox{ or } g_i:V\times X\to Z \,,
\end{eqnarray*}
for some Hilbert space $X$, 
can be cast into the form \eqref{mod}, \eqref{obs} by setting
\[
\eqalign{
f(t,\bu(t),\theta)=\tilde{f}(t,\bu(t),\theta(t))\,, \quad
g(t,\bu(t),\theta)=\tilde{g}(t,\bu(t),\theta(t))\,, \\
\cX=L^2(0,T;X)\times X_0\,.
}
\]
However, the growth conditions imposed below for $f$ and $g$ in order to guarantee well--definedness and differentiability of the involved operators would translate to quite different growth conditions on $\tilde{f}$, $\tilde{g}$ or rather would have to be derived explicitly. Thus, in most of this paper we actually have the case of time independent parameters in mind. 
\end{remark} 

\subsubsection*{Notation}

In defining iterative methods according to the standard literature, Hilbert space adjoints have to be used, that will be denoted by ${}^\star$ as opposed to the Banach space adjoints denoted by ${}^*$.

Inner products and dual pairings will be denoted by $(\cdot,\cdot)$ and $\langle\cdot,\cdot\rangle$, respectively, with subscripts indicating the corresponding spaces.
The Hilbert spaces $\cX$, $Z$, $H$ (that will typically be Lebesgue spaces in applications) will immediately be identified with their duals, especially when defining Banach space adjoints such as $f'_\theta(t,\bu(t),\theta)^*:\cW^*\to \cX^*=\cX$. For $V$ (while identifying it with its bidual $V^{**}$) we will make a notionally clear distinction from its dual $V^*$, as these are typically Sobolev spaces of positive and negative differentiability order, respectively, thus having different inner products.  
Thus, for defining adjoints, the Riesz isomorphisms on $V$ and $V^*$, will be useful.
\begin{eqnarray}
D:V\to V^* \quad \forall u,v\in V\,: \ \dup{Du}{v}{V}:=(u,v)_V
\label{D}\\
I:V^*\to V \quad \forall u^*,v^*\in V^*\,: \ \dup{v^*}{Iu^*}{V}:=(u^*,v^*)_{V^*}.
\label{I}
\end{eqnarray}
The notation $D$ and $I$ relates to the fact that in the context of parabolic differential equations these operators correspond to a spatial differential and antiderivative (``integration'') operator, respectively.
On the space $\cU=W^{1,2,2}(0,T;V,V^*)$, we use 
\begin{equation}\label{innprodU}
(\bu,\bv)_\cU=\int_0^T (\dot{\bu}(t)+D\bu(t),\dot{\bv}(t)+D\bv(t))_{V^*}dt+(\bu(0),\bv(0))_H
\end{equation}
as an inner product, which by standard results on well-posedness of the evolution equation 
$\dot{\bu}(t)+D\bu(t)=\bw(t)$, $t\in(0,T)$, $\bu(0)=u_0$, 
induces a norm that is equivalent to the standard norm 
$\sqrt{\int_0^T(\|\dot{\bu}(t)\|_{V^*}^2+\|\bu(t)\|_V^2)dt}$ on $\cU$.
Moreover, we will frequently make use of the integration by parts formula in Bochner spaces
\begin{equation}\label{intbyparts}
\eqalign{
\forall \bu,\bv\in \cU\, : \ (\bu(T),\bv(T))_H-(\bu(0),\bv(0))_H=\\
\hspace*{2cm}\int_0^T(\dup{\dot{\bu}(t)}{\bv(t)}{V}+\dup{\dot{\bv}(t)}{\bu(t)}{V})\, dt\,,
}
\end{equation}
and of continuity of the embedding $\cU\to C(0,T;H)$ 
\begin{equation}\label{cUC0TH}
\forall \bv\in \cU\, : \ \sup_{t\in(0,T)}\|\bv(t)\|_H\leq c_{\cU\to C(0,T;H)}\|\bv\|_\cU\,.
\end{equation}

\medskip

The remainder of this paper is organized as follows.
In Section \ref{sec:setting} we provide an appropriate function space setting for evolutionary problems, including the proofs of G\^{a}teaux differentiability and the derivation of adjoints for the involved forward operators. Section \ref{sec:methods} states six iterative regularization methods (Landweber, Landweber-Kaczmarz and the iteratively regularized Gauss-Newton method, each in their all-at-once and reduced version) in this dynamic setting and provides a comparison among them. In Section \ref{sec:convanal} we collect some results on convergence of these methods, by combining existing Hilbert  space regularization theory with appropriate assumptions tailored to the time dependent setting. Finally, in section \ref{sec:Concl} we give a short summary and an outlook on further related research questions.
 
\section{Function space setting for reduced and all-at-once setting} \label{sec:setting}
We first of all set the stage for formulating the iterative methods to be considered below by putting \eqref{mod}, \eqref{obs} into an appropriate function space setting and computing derivatives as well as their adjoints.

\subsection{All-at-once formulation}
A formulation containing both state $\bu$ and parameter $\theta$ as unknown is quite straightforward to derive and only requires very moderate assumptions on the functions $f$ and $g$, as long as one is willing to accept the a priori assumption of existence of a solution $(\budag,\thdag)$ to \eqref{modex}, \eqref{obsex}.

The forward operator 
\begin{equation}\label{bF}
\bF:\cU\times\cX\to \cW\times H\times\cY \,, \quad 
(\bu,\theta)\mapsto 
\left(\begin{array}{c}
\dot{\bu}-f(\cdot,\bu,\theta)\\
\bu(0)-u_0(\theta)\\
g(\cdot,\bu,\theta)
\end{array}\right),
\end{equation}
with $\cU,\cW,\cY$ as in \eqref{UWY} is well-defined, provided additionally to the above mentioned Caratheodory property, $f$ and $g$  satisfy the growth conditions 
\begin{eqnarray}
\hspace*{-2cm}\forall t\in(0,T),\ v\in V, \ \theta \in \cX \: \quad 
\|f(t,v,\theta)\|_{V^*}\leq \phi(\|v\|_H+\|\theta\|_\cX) (1+\|v\|_V), \label{growthf}\\
\hspace*{-2cm}\forall t\in(0,T),\ v\in V, \ \theta \in \cX \: \quad 
\|g(t,v,\theta)\|_{Z}\leq \psi(\|v\|_H+\|\theta\|_\cX) (1+\|v\|_V), \label{growthg}
\end{eqnarray}
with increasing functions $\phi$ and $\psi$. 
(In \eqref{bF} and further on below we use the simple notation $f(\cdot,\bu,\theta)$, $g(\cdot,\bu,\theta)$ for the action of the Nemitskii operators induced by the functions $f(\cdot,\cdot,\theta)$ and $g(\cdot,\cdot,\theta)$ on the function $\bu$, cf., e.g., \cite{AppellZabrejko90}.)
Therewith, the inverse problem of finding $(\bu^\dagger,\theta^\dagger)$ satisfying  \eqref{modex}, \eqref{obsex} from given perturbed data 
\[
\bby^\delta=(0,0,\by^\delta)\approx(0,0,\by)=\bby
\]
can be written as 
\[
\bF(\bu,\theta)=\bby \,.
\]  

The G\^{a}teaux derivative of $\bF$ can be easily computed and justified under appropriate assumptions on $f$ and $g$.

\begin{proposition} \label{prop:deriv_bF}
Let $f$, $g$, besides being Caratheodory mappings for every fixed $\theta\in\cX$, be G\^{a}teaux differentiable with respect to their second and third argument, for almost all $t\in(0,T)$, with linear and continuous derivatives and satisfy the following growth conditions:
\begin{equation}\label{growthfuftheta}
\hspace*{-2cm}
\eqalign{
f(t,v,\theta)=f^1(t,v,\theta)+f^2(t,v,\theta)\, \\
\|{f^1_u}'(t,v,\theta)\|_{V\to V^*}\leq \phi^1_u(\|v\|_H+\|\theta\|_\cX) \,, \\
\|{f^2_u}'(t,v,\theta)\|_{H\to V^*}\leq \phi_{H,u}^2(\|v\|_H+\|\theta\|_\cX) (1+\|v\|_V)\,, \\
\|f_\theta'(t,v,\theta)\|_{\cX\to V^*}\leq \phi_\theta(\|v\|_H+\|\theta\|_\cX) (1+\|v\|_V) \,, 
}
\end{equation}
\begin{equation}\label{growthgugtheta}
\hspace*{-2cm}
\eqalign{
g(t,v,\theta)=g^1(t,v,\theta)+g^2(t,v,\theta)\, \\
\|{g^1_u}'(t,v,\theta)\|_{V\to Z}\leq \psi^1_u(\|v\|_H+\|\theta\|_\cX) \,, \\
\|{g^2_u}'(t,v,\theta)\|_{H\to Z}\leq \psi_{H,u}^2(\|v\|_H+\|\theta\|_\cX) (1+\|v\|_V) \,, \\
\|g_\theta'(t,v,\theta)\|_{\cX\to Z}\leq \psi_\theta(\|v\|_H+\|\theta\|_\cX) (1+\|v\|_V) \,, 
}
\end{equation}
with increasing functions $\phi^1_u$, $\phi_{H,u}^2$, $\phi_\theta$, $\psi^1_u$, $\psi_{H,u}^2$, $\psi_\theta$.
Moreover, let also $u_0$ be G\^{a}teaux differentiable.

Then $\bF$ as defined by \eqref{bF} is G\^{a}teaux differentiable on $\cU\times\cX$ and its derivative is given by
\begin{equation}\label{bFprime}
\hspace*{-2cm}
\bF'(\bu,\theta):\cU\times\cX\to \cW\times H\times\cY \,, \quad 
\bF'(\bu,\theta)= 
\left(\begin{array}{cc}
\ddt-f'_u(\cdot,\bu,\theta)&-f'_\theta(\cdot,\bu,\theta)\\
\delta_0 & -u_0'(\theta)\\
g'_u(\cdot,\bu,\theta)&g'_\theta(\cdot,\bu,\theta)
\end{array}\right).
\end{equation}
\end{proposition}
\begin{proof}
We show G\^{a}teaux differentiability of $f$. (The proof for $g$ is analogous and the differentiability of $(\bu,\theta)\mapsto \bu(0)-u_0(\theta)$ is obvious.)

For any $(\bv,\xi)\in\cU\times\cX$ we have 
\[
\hspace*{-2cm}
\frac{1}{\eps}\|f(\cdot,\bu+\eps\bv,\theta+\eps\xi)-f(\cdot,\bu,\theta)
-\eps f_u'(\cdot,\bu,\theta)\bv-\eps f_\theta'(\cdot,\bu,\theta)\xi\|_\cW\\
=\left(\int_0^T r_\eps(t)^2\, dt\right)^\frac12\,,
\]
where for all $\eps\in(0,\bar{\eps}]$ 
\begin{equation}\label{est:reps}
\hspace*{-2cm}
\eqalign{
r_\eps(t)=
\left\|
\int_0^1 \Bigl(
(f_u'(t,\bu(t)+\lambda\eps\bv(t),\theta+\lambda\eps\xi)-f_u'(t,\bu(t),\theta))\bv(t) 
\right.
\\ \hspace*{2cm}
\left.
+(f_\theta'(t,\bu(t)+\lambda\eps\bv(t),\theta+\lambda\eps\xi)-f_\theta'(t,\bu(t),\theta))\xi
\Bigr)\,d\lambda
\right\|_{V^*}\\
\leq 4\Bigl(\phi^1_u(\bar{s})\|\bv(t)\|_V
+\phi_{H,u}^2(\bar{s})(1+\|\bu(t)\|_V+\bar{\eps}\|\bv(t)\|_V)\|\bv(t)\|_H\\
\hspace*{2cm}+\phi_\theta(\bar{s})(1+\|\bu(t)\|_V+\bar{\eps}\|\bv(t)\|_V)\|\theta\|_{\cX}
\Bigr)\\
\leq \bar{c}\Bigl(1+\|\bv(t)\|_V+\|\bu(t)\|_V\Bigr)
=:\bar{r}(t),
}
\end{equation}
with $\bar{c}$ and 
\[
\eqalign{
\bar{s}=\|\bu(t)\|_H+\bar{\eps}\|\bv(t)\|_H+\|\theta\|_\cX+\bar{\eps}\|\xi\|_\cX\\
\leq c_{\cU\to C(0,T;H)}(\|\bu\|_{\cU}+\bar{\eps}\|\bv\|_{\cU})+\|\theta\|_\cX+\bar{\eps}\|\xi\|_\cX
}
\]
cf. \eqref{cUC0TH}, depending only on $\|\bu\|_{\cU}$, $\|\bv\|_{\cU}$, $\|\theta\|_\cX$, $\|\xi\|_\cX$ and $\bar{\eps}$ but not on $\eps$ nor $t$.
Hence, based on square integrability of $t\mapsto\|\bv(t)\|_V$, $t\mapsto\|\bu(t)\|_V$, the function $\bar{r}$ is square integrable as well.\\
On the other hand, by differentiability of $f$, we have $r_\eps(t)\to0$ for  almost  every $t\in(0,T)$. Thus Lebesgue's Dominated Convergence Theorem yields 
\[
\int_0^T r_\eps(t)^2\,dt \to 0 \mbox{ as }\eps\to0\,,
\]
thus G\^{a}teaux differentiability follows.
Linearity of the derivative is obvious and its continuity follows similarly to the above, by employing Lebesgue's Dominated Convergence Theorem.
\end{proof}

The Hilbert space adjoint of $\bF'(\bu,\theta)$ can be computed as follows.

\begin{proposition} \label{prop:adj_bF}
Let the assumptions of Proposition \ref{prop:deriv_bF} be satisfied. 
Then the Hilbert space adjoint of $\bF'(\bu,\theta)$ is given by
\begin{equation}\label{bFadj}
\eqalign{
\bF'(\bu,\theta)^\star:\cW\times H\times\cY\to\cU\times\cX \,, \\
\bF'(\bu,\theta)^\star= 
\left(\begin{array}{ccc}
(\ddt-f'_u(\cdot,\bu,\theta))^\star&\delta_0^\star&g'_u(\cdot,\bu,\theta)^\star\\
-f'_\theta(\cdot,\bu,\theta)^\star&-u_0'(\theta)^\star&g'_\theta(\cdot,\bu,\theta)^\star
\end{array}\right),
}
\end{equation}
where
\begin{eqnarray*}
(\ddt-f'_u(\cdot,\bu,\theta))^\star\bw=\bu^\bw\\
f'_\theta(\cdot,\bu,\theta)^\star\bw=\int_0^Tf'_\theta(t,\bu(t),\theta)^*I\bw(t)\, dt\\
\delta_0^\star h=\bu^h\\
u_0'(\theta)^\star=u_0'(\theta)^*\\
g'_u(\cdot,\bu,\theta)^\star\bz=\bu^\bz\\
g'_\theta(\cdot,\bu,\theta)^\star\bz=\int_0^Tg'_\theta(t,\bu(t),\theta)^*\bz(t)\, dt\,,
\end{eqnarray*}
$\bu^\bw$, $\bu^h$, $\bu^\bz$ solve (along with some adjoint states $\bp^\bw$, $\bp^\bz$)
\begin{equation}\label{bubw}
\hspace*{-2cm}
\begin{cases}{
\dot{\bu}^\bw(t)+D\bu^\bw(t)=\bw(t)+I^{-1}\bp^\bw(t)\quad t\in(0,T)\,, \quad\bu^\bw(0)=\bp^\bw(0)\\
-\dot{\bp}^\bw(t)+D^*\bp^\bw(t)=-(D^*+f_u'(t,\bu(t),\theta)^*)I\bw(t)\quad t\in(0,T)\,, \quad\bp^\bw(T)=0
}\end{cases}
\end{equation}
\begin{equation}\label{bubz}
\hspace*{-2cm}
\begin{cases}{
\dot{\bu}^\bz(t)+D\bu^\bz(t)=I^{-1}\bp^\bz(t)\quad t\in(0,T)\,, \quad\bu^\bz(0)=\bp^\bz(0)\\
-\dot{\bp}^\bz(t)+D^*\bp^\bz(t)=g_u'(t,\bu(t),\theta)^*\bz(t)\quad t\in(0,T)\,, \quad\bp^\bz(T)=0
}\end{cases}
\end{equation}
\begin{equation}\label{buh}
\hspace*{-2cm}
\dot{\bu}^h(t)+D\bu^h(t)=0\quad t\in(0,T)\,, \quad\bu^h(0)=h\,,
\end{equation}
and 
\[
\eqalign{
f'_u(t,\bu(t),\theta)^*:V^{**}=V\to V^*, \quad f'_\theta(t,\bu(t),\theta)^*:V^{**}=V\to \cX^*=\cX, \\ 
g'_u(t,\bu(t),\theta)^*:Z^*=Z\to V^*, \quad g'_\theta(t,\bu(t),\theta)^*:Z^*=Z\to \cX^*=\cX, \\ u_0'(\theta)^*:H^*=H\to\cX^*=\cX
}
\] 
are the respective Banach space adjoints.
\end{proposition}
\begin{proof}
The adjoints of the derivatives with respect to $\theta$ can be easily seen by the definition of $I$ and the fact that the Bochner integral allows to exchange the order of integration with respect to time and taking the inner product.
 
The adjoints of the derivatives with respect to $u$ can be established as follows, using \eqref{intbyparts}.\\
For any $\bw\in\cW$ and all $\bv\in\cU$
\[
\hspace*{-2cm}
\eqalign{
(\bu^\bw,\bv)_\cU=
\int_0^T (\dot{\bu}^\bw(t)+D\bu^\bw(t),\dot{\bv}(t)+D\bv(t))_{V^*}\, dt + (\bu^\bw(0),\bv(0))_H\\
=\int_0^T (\bw(t)+I^{-1}\bp^\bw(t),\dot{\bv}(t)+D\bv(t))_{V^*}\, dt + (\bp^\bw(0),\bv(0))_H\\
=\int_0^T (\bw(t),\dot{\bv}(t)+D\bv(t))_{V^*}\, dt 
+\int_0^T \dup{\dot{\bv}(t)+D\bv(t)}{\bp^\bw(t)}{V} \, dt
+ (\bp^\bw(0),\bv(0))_H\\
=\int_0^T (\bw(t),\dot{\bv}(t)+D\bv(t))_{V^*}\, dt 
+\int_0^T \dup{-\dot{\bp}^\bw(t)+D^*\bp^\bw(t)}{\bv(t)}{V} \, dt
\\
=\int_0^T (\bw(t),\dot{\bv}(t)+D\bv(t))_{V^*}\, dt 
-\int_0^T \dup{(D^*+f_u'(t,\bu(t),\theta)^*)I\bw(t)}{\bv(t)}{V} \, dt
\\
=\int_0^T (\bw(t),\dot{\bv}(t)-f_u'(t,\bu(t),\theta)\bv(t))_{V^*}\, dt 
=(\bw,(\ddt-f_u'(\cdot,\bu(t),\theta))\bv)_\cW\,.
}
\]
For any $\bz\in\cY$ and all $\bv\in\cU$
\[
\hspace*{-2cm}
\eqalign{
(\bu^\bz,\bv)_\cU=
\int_0^T (\dot{\bu}^\bz(t)+D\bu^\bz(t),\dot{\bv}(t)+D\bv(t))_{V^*}\, dt + (\bu^\bz(0),\bv(0))_H\\
=\int_0^T (I^{-1}\bp^\bz(t),\dot{\bv}(t)+D\bv(t))_{V^*}\, dt + (\bp^\bz(0),\bv(0))_H\\
=\int_0^T \dup{\dot{\bv}(t)+D\bv(t)}{\bp^\bz(t)}{V}\, dt + (\bp^\bz(0),\bv(0))_H\\
=\int_0^T \dup{-\dot{\bp}^\bz(t)+D^*\bp^\bz(t)}{\bv(t)}{V}\, dt \\
=\int_0^T \dup{g_u'(t,\bu(t),\theta)^*\bz(t)}{\bv(t)}{V}\, dt \\
=\int_0^T (g_u'(t,\bu(t),\theta)\bv(t),\bz(t))_Z\, dt 
=(g_u'(\cdot,\bu,\theta)\bv,\bz)_\cY\,.
}
\]
For any $h\in H$ and all  $\bv\in\cU$
\[
\hspace*{-2cm}
\eqalign{
(\bu^h,\bv)_\cU=
\int_0^T (\dot{\bu}^h(t)+D\bu^h(t),\dot{\bv}(t)+D\bv(t))_{V^*}\, dt + (\bu^h(0),\bv(0))_H\\
=(h,\bv(0))_H =(h,\delta_0\bv)_H\,.
}
\]
\end{proof}

\begin{remark}\label{rem:weakformulation}
The requirement of carrying out forward--backward solutions \eqref{bubz}, \eqref{bubw} for the evaluation of adjoints arises due to the strength of the norm on $\cU$. In order to avoid this, one might consider using the weak solution concept
\[
\eqalign{
\bu\in L^2(0,T;V)\cap C(0,T;H^{\mbox{\footnotesize weak}})\mbox{ and }
\forall \bv\in W^{1,2,2}(0,T;V,V^*)\, : \\
\hspace*{-2cm}
\int_0^T (\dup{\dot{\bv}(t)}{\bu(t)}{V}+\dup{f(t,\bu(t),\theta)}{\bv(t)}{V})\, dt 
-(\bv(T),\bu(T))_H
=-(\bv(0),u_0(\theta))_H
}
\]
for the model equation \eqref{mod}, which is motivated by the integration by parts formula \eqref{intbyparts}. However, this would just transfer the difficulties to the adjoints acting on $\cW$, which would then be the dual of $W^{1,2,2}(0,T;V,V^*)$.
Also, when using the solution space $L^2(0,T;V)\cap C(0,T;H^{\mbox{\footnotesize weak}})$ (where the superscript ${}^{\mbox{\footnotesize weak}}$ indicates the weak topology on $H$) we would be forced to leave the Hilbert space setting.
\end{remark}

\begin{remark}\label{rem:obsdiscr_aao}
In case of discrete measurements \eqref{obsdiscr}, the forward operator is given by
\begin{equation}\label{bFdiscr}
\bF:\cU\times\cX\to \cW\times H\times\widetilde{\cY} \,, \quad 
(\bu,\theta)\mapsto 
\left(\begin{array}{c}
\dot{\bu}-f(\cdot,\bu,\theta)\\
\bu(0)-u_0(\theta)\\
g_1(\cdot,\bu,\theta)\\
\vdots\\
g_n(\cdot,\bu,\theta)\\
\end{array}\right)\,,
\end{equation}
with $\cU,\cW$ as in \eqref{UWY}, 
\[
\widetilde{\cY}=Z^n\,,
\]  
where $f$, $g$ are assume to satisfy the growth conditions \eqref{growthf} and 
\begin{equation}
\forall i\in\{1,\ldots,n\},\ h\in H, \ \theta \in \cX \: \quad 
\|g_i(h,\theta)\|_{Z}\leq \psi_i(\|v\|_H+\|\theta\|_\cX) \label{growthgdiscr}
\end{equation}
with increasing functions $\psi_i$. The derivative and its adjoint are given by 
\begin{equation}\label{bFprimediscr}
\hspace*{-2cm}
\bF'(\bu,\theta):\cU\times\cX\to \cW\times H\times\widetilde{\cY} \,, \quad 
\bF'(\bu,\theta)= 
\left(\begin{array}{cc}
\ddt-f'_u(\cdot,\bu,\theta)&-f'_\theta(\cdot,\bu,\theta)\\
\delta_0 & u_0'(\theta)\\
g'_{1u}(\cdot,\bu,\theta)&g'_{1\theta}(\cdot,\bu,\theta)\\
\vdots&\vdots\\
g'_{nu}(\cdot,\bu,\theta)&g'_{n\theta}(\cdot,\bu,\theta)\\
\end{array}\right)
\end{equation}
\begin{equation}\label{bFadjdiscr}
\hspace*{-2cm}\eqalign{
\bF'(\bu,\theta)^\star:\cW\times H\times\tilde{\cY}\to\cU\times\cX \,, \\
\bF'(\bu,\theta)^\star= 
\left(\begin{array}{ccccc}
(\ddt-f'_u(\cdot,\bu,\theta))^\star&\delta_0^\star&g'_{1u}(\cdot,\bu,\theta)^\star&\cdots&g'_{nu}(\cdot,\bu,\theta)^\star\\
-f'_\theta(\cdot,\bu,\theta)^\star&-u_0'(\theta)^\star&g'_{1\theta}(\cdot,\bu,\theta)^\star&\cdots&g'_{n\theta}(\cdot,\bu,\theta)^\star
\end{array}\right)\,,
}
\end{equation}
where $g'_{i\theta}(\cdot,\bu,\theta)^\star\bz_i=g'_{i\theta}(\cdot,\bu,\theta)^*\bz_i$,
$\sum_{i=1}^n g'_{iu}(\cdot,\bu,\theta)^\star\bz_i=\bu^\bz$, with $\bu^\bz$ as in \eqref{bubz}, while the adjoint state $\bp^\bz$ in \eqref{bubz} has to be redefined by
\[
\hspace*{-1.5cm}
\eqalign{
\bp^\bz(t)=0 \quad t\in[t_n,T]\\
\mbox{For } i=n:-1:1 \\
\hspace*{1cm}-\dot{\bp}^\bz(t)+D^*\bp^\bz(t)=0
\quad t\in(t_{i-1},t_i)\,, \quad
\bp^\bz(t_i)=\lim_{t\to t_i^+} \bp^\bz(t_i)+{g_i}_u'(\bu(t_i),\theta)^*\bz_i\,,
}
\]
where $\bp^\bz$ is to be considered as a left sided continuous function at the observation points, i.e, 
$\bp^\bz(t_i)=\lim_{t\to t_i^-}\bp^\bz(t)$. 
\end{remark}

\subsection{All-at-once system}
For the Kaczmarz approach we need to rewrite the inverse problem as a system of operator equations
\begin{equation}\label{sysbF}
\forall j\in\{0,\ldots,m-1\} \, : \ \bF_j(\bu,\theta)=\bby_j
\end{equation}
where we use a subdivision of the time interval 
\begin{equation}\label{subintervals}
[0,T]=\bigcup_{j\in \{1,\ldots,m\}}[\tau_{j-1},\tau_j]
\end{equation}
with breakpoints $0=\tau_0<\tau_1<\cdots\tau_{m-1}<\tau_m=T$ 
(not to be mistaken with the observation points in the discrete measurement case \eqref{obsdiscr}, cf. Remark \ref{rem:obsdiscr_aao})
to define 
\begin{equation}\label{bFj}
\bF_j:\cU\times\cX\to \cW_j\times\cY_j \,, \quad 
(\bu,\theta)\mapsto 
\left(\begin{array}{c}
\left.\Bigl(\dot{\bu}-f(\cdot,\bu,\theta)\Bigr)\right|_{(\tau_j,\tau_{j+1})}\\
\left.g(\cdot,\bu,\theta)\right|_{(\tau_j,\tau_{j+1})}
\end{array}\right)
\end{equation}
for $j\in\{1,\ldots,m-1\}$,
\begin{equation}\label{bF0}
\bF_0:\cU\times\cX\to \cW_0\times H\times\cY_0 \,, \quad 
(\bu,\theta)\mapsto 
\left(\begin{array}{c}
\left.\Bigl(\dot{\bu}-f(\cdot,\bu,\theta)\Bigr)\right|_{(0,\tau_{1})}\\
\bu(0)-u_0(\theta)\\
\left.g(\cdot,\bu,\theta)\right|_{(0,\tau_{1})}
\end{array}\right)\,,
\end{equation}
with $\cU$ as in \eqref{UWY}, 
\begin{equation}\label{WjYj}
\cW_j=L_2(\tau_j,\tau_{j+1};V^*)\,, \quad \cY_j=L_2(\tau_j,\tau_{j+1};Z)\,, \quad j\in\{0,1,\ldots,m-1\}\,,
\end{equation}
and 
\begin{equation}\label{yj}
\eqalign{
\bby_j^\delta=(0,\by_j^\delta)\approx(0,\by_j)=\bby_j \,, \ j\in\{1,\ldots,m-1\}\,, \\
\bby_j^\delta=(0,0,\by_j^\delta)\approx(0,0,\by_j)=\bby_j \,, \\ 
\by_j^\delta=\left.\by^\delta\right|_{(\tau_j,\tau_{j+1})}\,, \ 
\by_j=\left.\by\right|_{(\tau_j,\tau_{j+1})}.
}
\end{equation}
Note that continuity of $\bu$ over the breakpoints $\tau_j$ is guaranteed by the embedding $\cU\subseteq C(0,T;H)$.

The derivative of $\bF_j$ is obviously defined as in \eqref{bFprime}, by applying restrictions to the subinterval $(\tau_j,\tau_{j+1})$ in the first and last row (while skipping the middle row for $j\geq1$). For defining the adjoint of $\bF_j$, it is readily checked that we just have to modify the definition of $\bu^\bz$, $\bu^\bw$ accordingly:
\[
\eqalign{
\bF_0'(\bu,\theta)^*= 
\left(\begin{array}{ccc}
((\ddt-f'_u(\cdot,\bu,\theta))\vert_{(0,\tau_1)})^\star&\delta_0^\star&(g'_u(\cdot,\bu,\theta)\vert_{(0,\tau_1)})^\star\\
-(f'_\theta(\cdot,\bu,\theta)\vert_{(0,\tau_1)})^\star&-u_0'(\theta)^\star&(g'_\theta(\cdot,\bu,\theta\vert_{(0,\tau_1)})^\star
\end{array}\right)
\\
\bF_j'(\bu,\theta)^*= 
\left(\begin{array}{cc}
((\ddt-f'_u(\cdot,\bu,\theta))\vert_{(\tau_j,\tau_{j+1})})^\star&(g'_u(\cdot,\bu,\theta)\vert_{(\tau_j,\tau_{j+1})})^\star\\
-(f'_\theta(\cdot,\bu,\theta)\vert_{(\tau_j,\tau_{j+1})})^\star&(g'_\theta(\cdot,\bu,\theta)\vert_{(\tau_j,\tau_{j+1})})^\star
\end{array}\right)\,,
}
\]
where
\begin{eqnarray*}
((\ddt-f'_u(\cdot,\bu,\theta))\vert_{(\tau_j,\tau_{j+1})})^\star\bw_j=\bu^\bw_j\\
(f'_\theta(\cdot,\bu,\theta)\vert_{(\tau_j,\tau_{j+1})})^\star\bw_j=\int_{\tau_j}^{\tau_{j+1}}f'_\theta(t,\bu(t),\theta)^*I\bw_j(t)\, dt\\
\delta_0^\star h=\bu^h\\
u_0'(\theta)^\star=u_0'(\theta)^*\\
(g'_u(\cdot,\bu,\theta)\vert_{(\tau_j,\tau_{j+1})})^\star\bz_j=\bu^\bz_j\\
(g'_\theta(\cdot,\bu,\theta)\vert_{(\tau_j,\tau_{j+1})})^\star\bz_j=\int_{\tau_j}^{\tau_{j+1}}g'_\theta(t,\bu(t),\theta)^*\bz_j(t)\, dt
\end{eqnarray*}
$\bu_j^\bw$, $\bu^h$, $\bu_j^\bz$ solve
\begin{equation}\label{bubwj}
\hspace*{-2cm}
\begin{cases}{
\dot{\bu}^\bw_j(t)+D\bu^\bw_j(t)=\chi_j(\bw_j)(t)+I^{-1}\bp^\bw_j(t)\quad t\in(0,T)\,, \quad \bu^\bw_j(0)=\bp^\bw_j(0)\\
-\dot{\bp}^\bw_j(t)+D^*\bp^\bw_j(t)=-\chi_j((D^*+f_u'(\cdot,\bu,\theta)^*)I\bw_j)(t)\quad t\in(0,T)\,, \quad\bp^\bw_j(T)=0\,,
}\end{cases}
\end{equation}
\begin{equation}\label{bubzj}
\hspace*{-2cm}
\begin{cases}{
\dot{\bu}^\bz_j(t)+D\bu^\bz_j(t)=I^{-1}\bp^\bz_j(t)\quad t\in(0,T)\,, \quad\bu^\bz_j(0)=\bp^\bz_j(0)\\
-\dot{\bp}^\bz_j(t)+D^*\bp^\bz_j(t)=\chi_j(g_u'(\cdot,\bu,\theta)^*\bz_j)(t)\quad t\in(0,T)\,, \quad\bp^\bz_j(T)=0
}\end{cases}
\end{equation}
\begin{equation}\label{buh0}
\hspace*{-2cm}
\dot{\bu}^h(t)+D\bu^h(t)=0\quad t\in(0,\tau_1)\,, \quad\bu^h(0)=h\,,
\end{equation}
with the extension by zero operator $\chi_j$ defined by
\begin{equation}\label{chij}
\chi_j(\phi)(t)=\left\{\begin{array}{ll} \phi(t)\mbox{ for }t\in (\tau_j,\tau_{j+1})\\
0\mbox{ else.}\end{array}\right.
\end{equation}

\subsection{Reduced formulation}
This setting is based on elimination of the model equation by means of the parameter-to-state map 
\begin{equation}\label{S}
\hspace*{-2cm}S:\cX\to\cU\, \quad \theta\mapsto \bu \mbox{ solving }\ 
\dot{\bu}(t)=f(t,\bu(t),\theta) \ t\in(0,T)\,, \quad\bu(0)=u_0(\theta)\,.
\end{equation}
To this end, one needs to impose certain conditions on $f$, for example those from the following proposition.
\begin{proposition}\label{prop:Swelldef}
Let $\theta\in\cX$. Assume that 
\begin{itemize}
\item
for almost all $t\in(0,T)$, the mapping $-f(t,\cdot,\theta)$ be pseudomonotone, i.e., $f(t,\cdot,\theta)$ is bounded and 
\[
\hspace*{-2.5cm}\left.\begin{array}{r}
u_k\rightharpoonup u\\
\liminf_{k\to\infty}\dup{f(t,u_k,\theta)}{u_k-u}{V}\geq0
\end{array}\right\}\ \Rightarrow \
\left\{\begin{array}{l}
\forall v\in V \, : \ \dup{f(t,u,\theta)}{u-v}{V}\\
\geq\limsup_{k\to\infty}\dup{f(t,u_k,\theta)}{u_k-v}{V}
\end{array}\right.
\]
\item $-f(\cdot,\cdot,\theta)$ is semi-coercive, i.e., 
\[
\hspace*{-2.3cm}
\forall v\in V\,\forall^{\mbox{\footnotesize{a.e.}}}t\in(0,T)\, : \ 
-\dup{f(t,v,\theta)}{v}{V}\geq c_0^\theta |v|_V^2-c_1^\theta(t)|v|_V-c_2^\theta(t)\|v(t)\|_H^2
\]
for some $c_0^\theta>0$, $c_1^\theta\in L^2(0,T)$, $c_2^\theta\in L^1(0,T)$, and some seminorm $|\cdot|_V$ satisfying 
\[
\hspace*{-2.3cm}
\forall v\in V\, : \ 
\|v\|_V\leq c_{|\cdot|} (|v|_V+\|v\|_H)
\]
for some $c_{|\cdot|}>0$,
\item $f$ satisfies the growth condition \eqref{growthf} as well as
\[
\hspace*{-2.3cm}
\forall u,v\in V\,\forall^{\mbox{\footnotesize{a.e.}}}t\in(0,T)\, : \  
\dup{f(t,u,\theta)-f(t,v,\theta)}{u-v}{V}\leq \rho^\theta(t)\|u-v\|_H^2
\]
for some $\rho^\theta\in L_1(0,T)$.
\end{itemize}
Then $S(\theta)\in\cU$ is well defined.

\end{proposition}
\begin{proof}
\cite[Theorems 8.27, 8.31]{Roubicek}
\end{proof}

Therewith, under the conditions of Proposition \ref{prop:Swelldef}, as well as the growth condition \eqref{growthg} on $g$, the reduced forward operator  
\begin{equation}\label{F}
F:\cX\to \cY \,, \quad 
\theta\mapsto g(\cdot,S(\theta),\theta)
\end{equation}
with $\cY$ as in \eqref{UWY} is well-defined 
and the inverse problem of finding $(\bu^\dagger,\theta^\dagger)$ satisfying  \eqref{modex}, \eqref{obsex} from given perturbed data can be written as 
\begin{equation}\label{Fthetay}
F(\theta)=\by\approx\by^\delta \,,
\end{equation}
upon setting $\bu^\dagger=S(\theta^\dagger)$.

Differentiability of the forward operator is here mainly a question of differentiability of $S$, which relies on the sensitivity equation \eqref{sensitivityequation}.
\begin{proposition} \label{prop:deriv_F}
Let $S$ be well-defined by \eqref{S} and let $f$, $g$, besides being Caratheodory mappings for every fixed $\theta\in\cX$, be G\^{a}teaux differentiable with respect to their second and third argument, for almost all $t\in(0,T)$ and satisfy the growth conditions
\begin{eqnarray}
\|f_u'(t,v,\theta)\|_{V\to V^*}\leq c_{f,u} \,,
\label{fubounded}\\
\|f_\theta'(t,v,\theta)\|_{\cX\to V^*}\leq c_{f,\theta} (1+\|v\|_V)\,,
\label{fthetabounded}
\end{eqnarray}
and \eqref{growthgugtheta}.
Moreover, let also $u_0$ be G\^{a}teaux differentiable with derivative $u_0'$ uniformly bounded on bounded subsets of $\cX$.

Then $F$ as defined by \eqref{F} is G\^{a}teaux differentiable on $\cX$ and its derivative is given by
\begin{equation}\label{Fprime}
\hspace*{-2cm}{F}'(\theta):\cX\to \cY \,, \quad 
({F}'(\theta)\xi)(t)= 
g'_u(t,S(\theta)(t),\theta)\bu^\xi(t)+g'_\theta(t,S(\theta)(t),\theta)\xi\,,
\end{equation}
where $\bu^\xi=S'(\theta)\xi$ solves the sensitivity equation
\begin{equation}\label{sensitivityequation}
\hspace*{-2cm}\dot{\bu}^\xi(t)=f'_u(t,S(\theta)(t),\theta)\bu^\xi(t)+f'_\theta(t,S(\theta)(t),\theta)\xi \quad t\in(0,T)\,, \quad\bu^\xi(0)=u_0'(\theta)\xi\,.
\end{equation}
\end{proposition}
\begin{proof}
For $\theta,\xi\in\cX$ fixed and and $\eps\in(0,1]$, the functions $\bv_\eps=\frac{1}{\eps}(S(\theta+\eps\xi)-S(\theta))$ and $\widetilde{\bv}_\eps=\frac{1}{\eps}(S(\theta+\eps\xi)-S(\theta)-\bu^{\eps\xi})$ solve
\[
\eqalign{
\dot{\bv}_\eps(t)=A_\eps(t)\bv_\eps(t)+B_\eps(t)\xi \quad t\in(0,T)\,, \quad\bv(0)=v_{\eps,0}\,,\\
\dot{\widetilde{\bv}}_\eps(t)=\tilde{A}(t)\widetilde{\bv}_\eps(t)+b_\eps(t) \quad t\in(0,T)\,, \quad\widetilde{\bv}(0)=\widetilde{v}_{\eps,0}\,,
}
\]
with
\[
\eqalign{
A_\eps(t)=\int_0^1 f_u'(t,\lambda S(\theta+\eps\xi)(t)+(1-\lambda)S(\theta)(t),\theta+\lambda\eps\xi)\, d\lambda\,,\\
B_\eps(t)=\int_0^1 f_\theta'(t,\lambda S(\theta+\eps\xi)(t)+(1-\lambda)S(\theta)(t),\theta+\lambda\eps\xi)\, d\lambda\,,\\
\tilde{A}(t)=f_u'(t,\bu(t),\theta)\,,\\
b_\eps(t)=\frac{1}{\eps}\Bigl(f(t,S(\theta)(t)+\eps\bv_\eps(t),\theta+\eps\xi)-f(t,S(\theta)(t),\theta)\,,\\
\hspace*{2.5cm}-\eps f_u'(t,S(\theta)(t),\theta)\bv_\eps(t)-\eps f_\theta'(t,S(\theta)(t),\theta)\xi\Bigr)\,,\\
v_{\eps,0}=\frac{1}{\eps}(u_0(\theta+\eps\xi)-u_0(\theta))\,, \ 
\widetilde{v}_{\eps,0}=\frac{1}{\eps}(u_0(\theta+\eps\xi)-u_0(\theta)-\eps u_0'(\theta)\xi)\,,
}
\]
hence 
\[
\hspace*{-2cm}
\eqalign{
\bv_\eps(t)=e^{\int_0^t A_\eps(\sigma)\, d\sigma}v_{\eps,0}
+\int_0^t B_\eps(\tau)\xi\,d\tau
+\int_0^te^{\int_s^t A_\eps(\sigma)\, d\sigma} A_\eps(s) \int_0^s B_\eps(\tau)\xi\,d\tau\, ds\\
\widetilde{\bv}_\eps(t)=e^{\int_0^t A_\eps(\sigma)\, d\sigma}\widetilde{v}_{\eps,0}
+\int_0^t b_\eps(\tau)\,d\tau
+\int_0^te^{\int_s^t{\tilde{A}}(\sigma)\, d\sigma}{\tilde{A}}(s) \int_0^s b_\eps(\tau)\,d\tau\, ds\,.
}
\]
Thus, by 
\[
\hspace*{-2cm}\eqalign{
\|A_\eps(t)\|_{V\to V^*}\leq c_{f,u}\,,\quad
\|{\tilde{A}}(t)\|_{V\to V^*}\leq c_{f,u}\\
\|B_\eps(t)\|_{\cX\to V^*} \leq c_{f,\theta} (1+\|S(\theta)(t)\|_V+\eps\|\bv_\eps(t)\|_V)
}
\]
(note that at this point it becomes evident that we might only generalize the uniform bound $c_{f,u}$ in \eqref{fubounded} to an at most logarithmically growing function $\phi_u(\|v\|_H+\|\theta\|_\cX)$ )
we can estimate 
\[
\hspace*{-2cm}
\|\bv_\eps\|_\cU\leq c  (\|v_{\eps,0}\|_H+\|B_\eps(\cdot)\xi\|_{L^2(0,T;V^*)})
\leq c (\|v_{\eps,0}\|_H+ c_{f,\theta}(1+\|S(\theta)\|_{\cU}+\eps\|\bv_\eps\|_{\cU}))
\]
for any $\eps\in(0,1]$, thus
\begin{equation}\label{boundveps}
\|\bv_\eps\|_\cU\leq  2 c (\sup_{\eps\in(0,\bar{\eps}]}\|u_0'(\theta+\eps\xi)\xi\|_H+c_{f,\theta}(1+\|S(\theta)\|_{\cU}))
\end{equation}
for any $\eps\in(0,\bar{\eps}]$ with $\bar{\eps}=\min\{1,\frac{1}{2 c c_{f,\theta}}\}$.
This allows to proceed analogously to the proof of Proposition \ref{prop:deriv_bF} to obtain
\[
\|\widetilde{\bv}_\eps\|_\cU\leq c \|b_\eps\|_{L^2(0,T;V^*)} = c (\|\tilde{v}_{\eps,0}\|_H+\|r_\eps\|_{L^2(0,T)})\to0\mbox{ as }\eps\to0\,,
\]
where $r_\eps$ is defined as in \eqref{est:reps} with $\bu$ replaced by $S(\theta)$ and $\bv$ by $\bv_\eps$ (whose uniform boundedness in $\cU$ \eqref{boundveps} is therefore essential).

Having proven G\^{a}teaux differentiability of $S$ one can proceed analogously to the proof of Proposition \ref{prop:deriv_bF} (with $f$  replaced by $g$) to show that $F$ is G\^{a}teaux differentiable with derivative given by \eqref{Fprime}.
\end{proof}

\begin{proposition} \label{prop:adj_F}
Let the assumptions of Proposition \ref{prop:deriv_F} be satisfied. 
Then the Hilbert space adjoint of ${F}'(\theta)$ is given by
\begin{eqnarray}
\hspace*{-2cm}
{F}'(\theta)^\star:\cY\to\cX\,, \label{Fadj}\\
\hspace*{-2cm}
{F}'(\theta)^\star\bz=\int_0^T (g'_\theta(t,S(\theta)(t),\theta)^*\bz(t)+f'_\theta(t,S(\theta)(t),\theta)^*\bp^\bz(t))\,dt+u_0'(\theta)^*\bp^\bz(0)\,,\nonumber
\end{eqnarray}
where $\bp^\bz$ solves
\[
\hspace*{-2cm}-\dot{\bp}^\bz(t)=f'_u(t,S(\theta)(t),\theta)^*\bp^\bz(t)+g_u'(t,S(\theta)(t),\theta)^*\bz(t)\quad t\in(0,T)\,, \quad\bp^\bz(T)=0\,.
\]
\end{proposition}

\begin{proof}
For any $\bz\in\cY$ and all $\xi\in\cX$, using the interchangeability of time integral and inner product for Bochner spaces, we get for ${F}'(\theta)^\star\bz$ as defined in \eqref{Fadj},
\[
\hspace*{-2.5cm}\eqalign{
({F}'(\theta)^\star\bz,\xi)_\cX\\
=\int_0^T \Bigl(
(g'_\theta(t,S(\theta)(t),\theta)^*\bz(t),\xi)_\cX+(f'_\theta(t,S(\theta)(t),\theta)^*\bp^\bz(t),\xi)_\cX\Bigr)\,dt+(u_0'(\theta)^*\bp^\bz(0),\xi)_\cX\,,
}
\]
where
\[
(g'_\theta(t,S(\theta)(t),\theta)^*\bz(t),\xi)_\cX
= (g'_\theta(t,S(\theta)(t),\theta)\xi,\bz(t))_Z\,,
\]
and, due to \eqref{intbyparts},
\[
\eqalign{
\int_0^T 
\dup{f'_\theta(t,S(\theta)(t),\theta)\xi}{\bp^\bz(t)}{V}\,dt+(u_0'(\theta)\xi,\bp^\bz(0))_H\\
=\int_0^T 
\dup{\dot{\bu}^\bz(t)-f'_u(t,S(\theta)(t),\theta)\bu^\bz}{\bp^\bz(t)}{V}\,dt+(\bu^\bz(0),\bp^\bz(0))_H\\
=\int_0^T 
\dup{-\dot{\bp}^\bz(t)-f'_u(t,S(\theta)(t),\theta)^*\bp^\bz}{\bu^\bz(t)}{V}\,dt\\
=\int_0^T 
\dup{g_u'(t,S(\theta)(t),\theta)^*\bz(t)}{\bu^\bz(t)}{V}\,dt\\
=\int_0^T 
(g_u'(t,S(\theta)(t),\theta)\bu^\bz(t),\bz(t))_Z\,dt\,,
}
\]
hence altogether
\[
\eqalign{
({F}'(\theta)^\star\bz,\xi)_\cX
=\int_0^T(g'_\theta(t,S(\theta)(t),\theta)\xi+g_u'(t,S(\theta)(t),\theta)\bu^\bz(t),\bz(t))_Z\,dt\\
=({F}'(\theta)\xi,\bz)_\cY
}
\]
\end{proof}

\begin{remark}\label{rem:obsdiscr_red}
Similarly to Remark \ref{rem:obsdiscr_aao}, in case of discrete measurements \eqref{obsdiscr}, the adjoint state $\bp^\bz$ has to be redefined by 
\[
\eqalign{
\bp^\bz(t)=0 \quad t\in[t_n,T]\\
\mbox{For } i=n:-1:1 \\
\hspace*{1cm}-\dot{\bp}^\bz(t)=f'_u(t,S(\theta)(t),\theta)^*\bp^\bz(t)
\quad t\in(t_{i-1},t_i)\,, \\
\hspace*{3cm}
\bp^\bz(t_i)=\lim_{t\to t_i^+} \bp^\bz(t_i)+{g_i}_u'(S(\theta)(t_i),\theta)^*\bz_i\,,
}
\]
where again $\bp^\bz$ is a left sided continuous function at the observation points. 
\end{remark}

\subsection{Reduced system}
Using the subdivision \eqref{subintervals}, we can again replace \eqref{Fthetay} by a system of $m$ equations
\begin{equation}\label{sysF}
\forall j\in\{0,\ldots,m-1\} \, : \ F_j(\theta)=\by_j\,,
\end{equation}
where  
\begin{equation}\label{Fj}
F_j:\cX\to \cY_j \,, \quad 
\theta\mapsto 
\left.g(\cdot,S(\theta),\theta)\right|_{(\tau_j,\tau_{j+1})}
\end{equation}
for $j\in\{0,\ldots,m-1\}$,
with $\cY_j$ as in \eqref{WjYj} and $\by_j^\delta$, $\by_j$ as in \eqref{yj}.

The derivative of $F_j$ is defined by restriction to $(\tau_j,\tau_{j+1})$ in \eqref{Fprime}, while the sensitivity equation \eqref{sensitivityequation} still has to be solved on the whole time interval $(0,T)$.
The Hilbert space adjoint of $F_j'(\theta)$ is given by $F_j'(\theta)^\star:\cY_j\to\cX$ 
\[
\hspace*{-2cm}
F_j'(\theta)^\star\bz_j=\int_{\tau_j}^{\tau_{j+1}} g'_\theta(t,S(\theta)(t),\theta)^*\bz_j(t)\, dt+ \int_0^Tf'_\theta(t,S(\theta)(t),\theta)^*\bp^\bz_j(t))\,dt+u_0'(\theta)^*\bp^\bz_j(0)\,,
\]
where $\bp^\bz_j$ solves
\[
\hspace*{-2cm}
-\dot{\bp}^\bz(t)=f'_u(t,S(\theta)(t),\theta)^*\bp^\bz(t)+\chi_j(g_u'(\cdot,S(\theta),\theta)^*\bz_j)(t)\quad t\in(0,T)\,, \quad\bp^\bz(T)=0\,.
\]

\section{Formulation and comparison of iterative regularization methods} \label{sec:methods}

In this subsection, we will first of all make explicit what one step of 
\begin{itemize}
\item the all-at-once Landweber method (aLW),
\[
(\bu_{k+1},\theta_{k+1})
=(\bu_k,\theta_k)-\mu_k \bF'(\bu_k,\theta_k)^\star(\bF(\bu_k,\theta_k)-\bby^\delta)
\]
\item the all-at-once Landweber-Kaczmarz method (aLWK),
\[
(\bu_{k+1},\theta_{k+1})
=(\bu_k,\theta_k)-\mu_k \bF_{j(k)}'(\bu_k,\theta_k)^\star(\bF_{j(k)}(\bu_k,\theta_k)-\bby_{j(k)}^\delta)
\]
\item the all-at-once iteratively regularized Gauss-Newton method (aIRGNM),
\[
\eqalign{
\bF'(\bu_k,\theta_k)^\star\Bigl(\bF'(\bu_k,\theta_k)((\bu_{k+1},\theta_{k+1})-(\bu_k,\theta_k))+\bF(\bu_k,\theta_k)-\bby^\delta\Bigr)\\
\hspace*{6.5cm}+\alpha_k ((\bu_{k+1},\theta_{k+1})-(\bar{\bu},\bar{\theta}))=0
}
\]
\item the reduced Landweber method (rLW),
\[
\theta_{k+1}
=\theta_k-\mu_k F'(\theta_k)^\star(F(\theta_k)-\by^\delta)
\]
\item the reduced Landweber-Kaczmarz method (rLWK),
\[
\theta_{k+1}
=\theta_k-\mu_k F_{j(k)}'(\theta_k)^\star(F_{j(k)}(\theta_k)-\by_{j(k)}^\delta)
\]
\item the reduced iteratively regularized Gauss-Newton method (rIRGNM),
\[
F'(\theta_k)^\star\Bigl(F'(\theta_k)(\theta_{k+1}-\theta_k)+F(\theta_k)-\by^\delta\Bigr)
+\alpha_k (\theta_{k+1}-\bar{\theta})=0
\]
\end{itemize}
means in the context of the dynamical state space system \eqref{mod}, \eqref{obs}.
Here 
\begin{itemize}
\item[$\circ$] $\mu_k\in(0,\frac{1}{\|\bF'(\bu_k,\theta_k)\|^2}]$ is a stepsize parameter, which can be chosen constant if $\|\bF'(\bu,\theta)\|$ is uniformly bounded, 
\item[$\circ$] $\alpha_k$ a regularization parameter, which can be chosen a priori as $\alpha_k=\alpha_0 q^k$ for some $q\in(0,1)$, 
\item[$\circ$] $(\bar{\bu},\bar{\theta})$ an a priori guess, and 
\item[$\circ$] $j(k)=k-p[\frac{k}{p}]$.
\end{itemize}

To this end, note that the evolutions involving $D$ and $D^*$ in \eqref{bubw}--\eqref{buh} 
\[
\eqalign{
\dot{\bv}(t)+D\bv(t)=\widetilde{\bw}(t)\quad t\in(0,T)\,, \quad\bv(0)=v_0\\
-\dot{\bp}(t)+D^*\bp(t)=\widetilde{\widetilde{\bw}}(t)\quad t\in(0,T)\,, \quad\bp(T)=p_T
}
\]
are usually much easier to solve than the nonlinear and linear ones driven by $f$ and its derivative. We emphasize this fact also notionally, by using the variation of constants formulas
\begin{equation}\label{varofconst}
\eqalign{
\bv(t)=e^{-tD}v_0+\int_0^t e^{(s-t)D}\widetilde{\bw}(s)\,ds\,, \\
\bp(t)=e^{(t-T)D^*}p_T+\int_t^T e^{(t-s)D^*}\widetilde{\widetilde{\bw}}(s)\,ds
}
\end{equation}

\subsection{all-at-once Landweber}\label{sec:aaoLW}
Set \quad $\bw_k(t)=\dot{\bu}_k(t)-f(t,{\bu}_k(t),\theta_k)$, \\
\hspace*{1cm} $h_k=\bu_k(0)-u_0(\theta_k)$,\\
\hspace*{1cm} $\bz_k(t)=g(t,{\bu}_k(t),\theta_k)-\by^\delta(t)$ \quad
(so that $(\bw_k,h_k,\bz_k)=\bF(\bu_k,\theta_k)-\bby^\delta$) \\
Evaluate the adjoint state
\[
\hspace*{-2cm}
\bp_k(t)=\int_t^T e^{(t-s)D^*}\Bigl(-(D^*+f_u'(s,\bu_k(s),\theta_k)^*)I\bw_k(s)+g_u'(s,\bu_k(s),\theta_k)^*\bz_k(s)\Bigr)\,ds
\]
and the direct state
\[
\hspace*{-2cm}
\bv_k(t)=e^{-tD}(\bp_k(0)+h_k)+\int_0^t e^{(s-t)D}(\bw_k(s)+I^{-1}\bp_k(s))\,ds
\] 
Set \quad $\bu_{k+1}=\bu_k-\mu_k\bv_k$,\\
\hspace*{1cm} $\theta_{k+1}=\theta_k-\mu_k\Bigl(\int_0^T(-f'_\theta(t,\bu_k(t),\theta_k)^*I\bw_k(t)+g'_\theta(t,\bu_k(t),\theta_k)^*\bz_k(t))\, dt-u_0'(\theta_k)^*h_k\Bigr)$.

\subsection{all-at-once Landweber-Kaczmarz}\label{sec:aaoLWK}
Set \quad $\bw_k(t)=\left.\Bigr(\dot{\bu}_k(t)-f(t,{\bu}_k(t),\theta_k)\Bigr)\right|_{(\tau_{j(k)},\tau_{j(k)+1})}$,\\[-0.3ex] 
\hspace*{1cm} $h_k=\bu_k(0)-u_0(\theta_k)$,\\
\hspace*{1cm} $\bz_k(t)=\left.\Bigl(g(t,{\bu}_k(t),\theta_k)-\by^\delta(t)\Bigr)\right|_{(\tau_{j(k)},\tau_{j(k)+1})}$,\\[-0.3ex]
(so that $(\bw_k,h_k,\bz_k)=\bF_{j(k)}(\bu_k,\theta_k)-\bby^\delta_{j(k)}$).\\
Evaluate the adjoint state
\[
\hspace*{-4cm}
\bp_k(t)=
\left\{\begin{array}{ll}
0 &\mbox{ for }t\geq\tau_{j(k)+1}\\
\int_t^{\tau_{j(k)+1}} e^{(t-s)D^*}
\Bigl(
-(D^*+f_u'(\cdot,\bu_k,\theta_k)^*I\bw_k+g_u'(\cdot,\bu_k,\theta_k)^*\bz_k
\Bigr)(s)\,ds & \mbox{ for }\tau_{j(k)}<t<\tau_{j(k)+1}\\
\bp_k(\tau_{j(k)}) & \mbox{ for }t<\tau_{j(k)}
\end{array}\right.
\]
and the direct state
\[
\hspace*{-2cm}
\bv_k(t)=e^{-tD}(\bp_k(0)+h_k)+\int_0^t e^{(s-t)D}(\chi_{j(k)}(\bw_k)(s)+I^{-1}\bp_k(s))\,ds
\] 
Set \quad $\bu_{k+1}=\bu_k-\mu_k\bv_k$,\\
\hspace*{1cm} $\theta_{k+1}=\theta_k-\mu_k\Bigl(\int_{\tau_{j(k)}}^{\tau_{j(k)+1}}(-f'_\theta(t,\bu_k(t),\theta_k)^*I\bw_k(t)+g'_\theta(t,\bu_k(t),\theta_k)^*\bz_k(t))\, dt-u_0'(\theta_k)^*h_k\Bigr)$.

\subsection{all-at-once iteratively regularized Gauss-Newton}\label{sec:aaoIRGNM}
Solve the coupled linear system
\begin{equation}\label{aaoIRGNM}
\hspace*{-2cm}
\eqalign{
\bw_k(t)=\dot{\bu}_k(t)-f(t,{\bu}_k(t),\theta_k)-f_u'(t,{\bu}_k(t),\theta_k)({\bu}_{k+1}(t)-{\bu}_k(t))-f_\theta'(t,{\bu}_k(t),\theta_k)(\theta_{k+1}-\theta_k)\\
h_k=\bu_k(0)-u_0(\theta_k)-u_0'(\theta_k)(\theta_{k+1}-\theta_k)\\
\bz_k(t)=g(t,{\bu}_k(t),\theta_k)+g_u'(t,{\bu}_k(t),\theta_k)({\bu}_{k+1}(t)-{\bu}_k(t))+g_\theta'(t,{\bu}_k(t),\theta_k)(\theta_{k+1}-\theta_k)-\by^\delta(t)\\
\bp_k(t)=\int_t^Te^{(t-s)D^*}\Bigl(-(D^*+f_u'(s,\bu_k(s),\theta_k)^*I\bw_k(s)+g_u'(s,\bu_k(s),\theta_k)^*\bz_k(s)\Bigr)\, ds\\
\alpha_k({\bu}_{k+1}(t)-\bar{\bu}(t))+e^{-tD}(\bp(0)+h_k)+\int_0^te^{(s-t)D}(\bw_k(s)+I^{-1}\bp_k(s))\,ds=0\\
\alpha_k(\theta_{k+1}-\bar{\theta})+
\int_0^T(-f'_\theta(t,\bu_k(t),\theta_k)^*I\bw_k(t)+g'_\theta(t,\bu_k(t),\theta_k)^*\bz_k(t))\, dt-u_0'(\theta_k)^*h_k=0
}
\end{equation}
for $(\bw_k,h_k,\bz_k,\bp_k,\bu_{k+1},\theta_{k+1})\in\cW\times H\times\cY\times\cU\times\cU\times\cX$.
\\[1ex]
(Note that here
$(\bw_k,h_k,\bz_k)=\bF(\bu_k,\theta_k)+\bF(\bu_k,\theta_k)({\bu}_{k+1}-\bar{\bu},\theta_{k+1}-\bar{\theta})-\bby^\delta$.)

\begin{remark}
With the operators
\[ 
\eqalign{
\cI^D:\cW\to\cU\,, \quad (\cI^D\bw)(t)= \int_0^te^{(s-t)D}\bw(s)\, ds\\
\cI_k^{f,D^*}:\cW\to\cU\,, \quad (\cI^{f,D^*}\bw)(t)= \int_t^Te^{(t-s)D^*}(D^*+f_u'(s,\bu_k(s),\theta_k)^*I\bw(s)\, ds\\
\cI_k^{g,D^*}:\cY\to\cU\,, \quad (\cI^{g,D^*}\bz)(t)= \int_t^Te^{(t-s)D^*}g_u'(s,\bu_k(s),\theta_k)^*\bz(s)\, ds\\
\cI_k^{f}:\cW\to\cX\,, \quad (\cI^{f}\bw)(t)= \int_0^Tf_\theta'(t,\bu_k(t),\theta_k)^*I\bw(t)\, dt\\
\cI_k^{g}:\cY\to\cX\,, \quad (\cI^{g}\bz)(t)= \int_0^Tg_\theta'(t,\bu_k(t),\theta_k)^*\bz(t)\, dt\\
K_k^f:\cU\to\cW\,, \quad K_k^f=f_u'(\cdot,\bu_k,\theta_k)\\
K_k^g:\cU\to\cY\,, \quad K_k^g=g_u'(\cdot,\bu_k,\theta_k)\\
L_k^f:\cX\to \cW\,, \quad L_k^f=f_\theta'(\cdot,\bu_k,\theta_k)\\
L_k^g:\cX\to \cY\,, \quad L_k^g=g_\theta'(\cdot,\bu_k,\theta_k)
}
\]
the system \eqref{aaoIRGNM} can be rewritten as
{\small
\begin{equation}\label{aaoIRGNMmat6}
\hspace*{-3cm}
\left(\begin{array}{cccccc}
\mbox{id}&0&0&0&K_k^f&L_k^f\\
0&\mbox{id}&0&0&0&u_0'(\theta_k)\\
0&0&\mbox{id}&0&-K_k^g&-L_k^g\\
\cI_k^{f,D^*}&0&-\cI_k^{g,D^*}&\mbox{id}&0&0\\
\cI^D&e^{-\cdot D}&0&\cI^DI^{-1}+e^{-\cdot D}\delta_0&\alpha_k\mbox{id}&0\\
\cI_k^{f}&-u_0'(\theta_k)^*&\cI_k^{g}&0&0&\alpha_k\mbox{id}
\end{array}\right)
\left(\begin{array}{c}
\bw_k\\ h_k\\ \bz_k\\ \bp_k\\ \bu_{k+1}-\bu_{k}\\ \theta_{k+1}-\theta_{k}\end{array}\right)
=
\left(\begin{array}{c}
\dot{\bu}_k-f(\cdot,\bu_k,\theta_k)\\ \bu_k(0)-u_0(\theta_k)\\ g(\cdot,\bu_k,\theta_k)-\by^\delta\\ 0\\ \alpha_k(\bar{\bu}-\bu_{k})\\ \alpha_k(\bar{\theta}-\theta_{k})\end{array}\right)
\end{equation}
}

Alternatively, re-inserting $\bw_k$ and $h_k$ from the first two lines in \eqref{aaoIRGNM} into the rest of the equations, we arrive at
\begin{equation}\label{aaoIRGNMmat4}
\hspace*{-2.5cm}
\eqalign{
\left(\begin{array}{cccc}
\mbox{id}&0&-K_k^g&-L_k^g\\
-\cI_k^{g,D^*}&\mbox{id}&-\cI_k^{f,D^*}K_k^f&-\cI_k^{f,D^*}L_k^f\\
0&\cI^DI^{-1}+e^{-\cdot D}\delta_0&\alpha_k\mbox{id}-\cI^DK_k^f&-\cI^DL_k^f-e^{-\cdot D}u_0'(\theta_k)\\
\cI_k^{g}&0&-\cI_k^{f}K_k^f&\alpha_k\mbox{id}+\cI_k^{f}L_k^f+u_0'(\theta_k)^*u_0'(\theta_k)
\end{array}\right)
\left(\begin{array}{c}
\bz_k\\ \bp_k\\ \bv_k\\ \theta_{k+1}-\theta_{k}\end{array}\right)\\
\hspace*{3cm}=\left(\begin{array}{c}
g(\cdot,\bu_k,\theta_k)-\by^\delta\\ 
\cI_k^{f,D^*}(\dot{\bu}_k-f(\cdot,\bu_k,\theta_k))\\ 
\alpha_k(\bar{\bu}-\bu_k)+\cI^D(\dot{\bu}_k-f(\cdot,\bu_k,\theta_k))+e^{-\cdot D}(\bu_k(0)-u_0(\theta_k))\\ 
\alpha_k(\bar{\theta}-\theta_{k})-\cI_k^f(\dot{\bu}_k-f(\cdot,\bu_k,\theta_k))-u_0'(\theta_k)^*(\bu_k(0)-u_0(\theta_k))\end{array}\right)
}
\end{equation}

\end{remark}

\subsection{reduced Landweber}\label{sec:redLW}
Solve the nonlinear evolution
\[
\hspace*{-2cm}\dot{\bu}_k(t)=f(t,\bu_k(t),\theta_k) \quad t\in(0,T)\,, \quad\bu_k(0)=u_0(\theta)
\]
for $\bu_k$ and set $\bz_k(t)=g(t,{\bu}_k(t),\theta_k)-\by^\delta(t)$\\
(so that $\bu_k=S(\theta_k)$, $\bz_k=F(\theta_k)-\by^\delta$).\\
Solve the adjoint equation
\[
\hspace*{-2cm}
-\dot{\bp}^\bz_k(t)=f'_u(t,{\bu}_k(t),\theta_k)^*\bp^\bz_k(t)+g_u'(t,{\bu}_k(t),\theta_k)^*\bz_k(t)\quad t\in(0,T)\,, \quad\bp^\bz_k(T)=0
\]
Set $\theta_{k+1}=\theta_k-\mu_k\Bigl(
\int_0^T (g'_\theta(t,{\bu}_k(t),\theta_k)^*\bz_k(t)+f'_\theta(t,{\bu}_k(t),\theta_k)^*\bp^\bz_k(t))\,dt+u_0'(\theta_k)^*\bp^\bz_k(0)\Bigr)$.

\subsection{reduced Landweber-Kaczmarz}\label{sec:redLWK}
Solve the nonlinear evolution
\[
\hspace*{-2cm}\dot{\bu}_k(t)=f(t,\bu_k(t),\theta_k) \quad t\in(0,T)\,, \quad\bu_k(0)=u_0(\theta)
\]
for $\bu_k$ and set $\bz_k(t)=g(t,{\bu}_k(t),\theta_k)-\by^\delta(t)\vert_{(\tau_{j(k)},\tau_{j(k)+1})}$\\
(so that $\bu_k=S(\theta_k)$, $\bz_k=F_{j(k)}(\theta_k)-\by^\delta$).\\
Solve the adjoint equation
\[
\hspace*{-2cm}
-\dot{\bp}^\bz_k(t)=f'_u(t,{\bu}_k(t),\theta_k)^*\bp^\bz_k(t)+\chi_{j(k)}(g_u'(,{\bu}_k,\theta_k)^*\bz_k)(t)\quad t\in(0,T)\,, \quad\bp^\bz_k(T)=0
\]
Set $\theta_{k+1}=\theta_k-\mu_k\Bigl(
\int_{\tau_{j(k)}}^{\tau_{j(k)+1}} g'_\theta(t,S{\bu}_k(t),\theta_k)^*\bz_k(t)\, dt+\int_0^T f'_\theta(t,{\bu}_k(t),\theta_k)^*\bp^\bz_k(t))\,dt+u_0'(\theta_k)^*\bp^\bz_k(0)\Bigr)$.

\subsection{reduced iteratively regularized Gauss-Newton}\label{sec:redIRGNM}
Solve the nonlinear evolution
\[
\hspace*{-2cm}\dot{\bu}_k(t)=f(t,\bu_k(t),\theta_k) \quad t\in(0,T)\,, \quad\bu_k(0)=u_0(\theta)
\]
for $\bu_k$.\\
Solve the linear system
\begin{equation}\label{redIRGNM}
\hspace*{-2cm}
\eqalign{
\bz_k(t)=g(t,{\bu}_k(t),\theta_k)+g_u'(t,{\bu}_k(t),\theta_k){\bv}_k(t)+g_\theta'(t,{\bu}_k(t),\theta_k)(\theta_{k+1}-\theta_k)-\by^\delta(t)\\
-\dot{\bp}^\bz_k(t)=f'_u(t,{\bu}_k(t),\theta_k)^*\bp^\bz_k(t)+g_u'(t,{\bu}_k(t),\theta_k)^*\bz_k(t)\quad t\in(0,T)\,, \quad\bp^\bz_k(T)=0\\
\dot{\bv}_k(t)=f'_u(t,{\bu}_k(t),\theta_k)\bv_k(t)+f'_\theta(t,{\bu}_k(t),\theta_k)(\theta_{k+1}-\theta_k) \quad t\in(0,T)\,, \\
\hspace*{9cm}\bv_k(0)=u_0'(\theta)(\theta_{k+1}-\theta_k)\\
\alpha_k(\theta_{k+1}-\bar{\theta})\\
+\int_0^T (g'_\theta(t,S{\bu}_k(t),\theta_k)^*\bz_k(t)+f'_\theta(t,{\bu}_k(t),\theta_k)^*\bp^\bz_k(t))\,dt+u_0'(\theta_k)^*\bp^\bz_k(0)=0
}
\end{equation}
for $(\bz_k,\bp_k,\bv_k,\theta_{k+1})\in\cY\times\cU\times\cU\times\cX$.\\
(Note that here $\bu_k=S(\theta_k)$, $\bv_k=S'(\theta_k)(\theta_{k+1}-\theta_k)$, $\bz_k=F(\theta_k)+F'(\theta_k)(\theta_{k+1}-\theta_k)$.)

\begin{remark}
With the operators
\[ 
\eqalign{
\widetilde{\cI}_k^f:\cX\to\cU\,, \quad \bv=\widetilde{\cI}_k^f\xi\mbox{ solves }\\
\dot{\bv}(t)=f'_u(t,\bu_k(t),\theta)\bv(t)+f'_\theta(t,\bu_k(t),\theta)\xi \quad t\in(0,T)\,, \quad\bv(0)=u_0'(\theta)\xi\\
\widetilde{\cI}_k^{fg}:\cY\to\cU\,, \quad \bp=\widetilde{\cI}_k^{fg}\bz\mbox{ solves }\\
-\dot{\bp}(t)=f'_u(t,\bu_k(t),\theta)^*\bp(t)+g'_u(t,\bu_k(t),\theta)^*\bz(t) \quad t\in(0,T)\,, \quad\bp(T)=0
}
\]
the system \eqref{redIRGNM} can be rewritten as
\begin{equation}\label{redIRGNMmat4}
\hspace*{-2.5cm}
\left(\begin{array}{cccc}
\mbox{id}&0&-K_k^g&-L_k^g\\
-\widetilde{\cI}_k^{fg}&\mbox{id}&0&0\\
0&0&\mbox{id}&-\widetilde{\cI}_k^f\\
\cI_k^{g}&\cI_k^{f}I^{-1}+u_0'(\theta_k)^*\delta_0&0&\alpha_k\mbox{id}
\end{array}\right)
\left(\begin{array}{c}
\bz_k\\ \bp_k\\ \bv_k\\ \theta_{k+1}-\theta_{k}\end{array}\right)
=
\left(\begin{array}{c}
g(\cdot,\bu_k,\theta_k)-\by^\delta\\ 0\\ 0\\ \alpha_k(\bar{\theta}-\theta_{k})\end{array}\right)
\end{equation}
\end{remark}
Unique solvability of the systems \eqref{aaoIRGNMmat4}, \eqref{redIRGNMmat4} follows from general well-posedness of the aIRGNM and the rIRGNM according to existing Hilbert space regularization theory.

\medskip

Comparing the reduced and the all-at-once versions of the three methods (LW, LWK, IRGNM) we observe the following. 
\begin{itemize}
\item All reduced versions require one solution of the nonlinear model per step, whereas no nonlinear model is solved in the all-at-once-versions.
\item In aLW instead of the nonlinear model we solve a simple linear evolution (involving the constant operator $D$); moreover, both rLW and aLW involve solving a linear backwards evolution, which is again a simple one in case of aLW.
\item The Kaczmarz version aLWK allows to restrict not only the observations but also the adjoint state computation to the subinterval $(\tau_j,\tau_{j+1})$ of  the current step, whereas in rLWK in each step a full (nonlinear) forward and backward sweep over the whole time interval $(0,T)$ has to be done.
\item Both aIRGNM and rIRGNM lead to coupled systems (in the latter case, after solving a forward model); the one for the aIRGNM as compared to rIRGNM at a first glance either involves more unknown functions (\eqref{aaoIRGNMmat6} versus \eqref{redIRGNMmat4}) or leads to a denser, i.e., more strongly coupled system (\eqref{aaoIRGNMmat4} versus \eqref{redIRGNMmat4}). However, again, note that in the all-at-once version only very simple forward and backward evolutions are involved, whereas in the reduced version we have to deal with the linearized model and its adjoint which lead to different (and possibly nonautonomous) systems in each step. 
\end{itemize}

Potential computational advantages of the all-at once versions are expected to be more pronounced for nonlinear models. However, tests of their numerical performance for several practically relevant examples are yet to be done and will be subject of future work.

\section{Convergence analysis}\label{sec:convanal}

For all methods under consideration we can apply the existing regularization theory in Hilbert spaces, once we have quantified the noise level $\delta$ (whose convergence to zero yields convergence of these iterative methods with an appropriate choice of the stopping index in dependence of $\delta$) and under appropriate conditions on the forward operator. 

\subsection{Noise from perturbations}
Note that the difference between the given data $\bby^\delta$ and $\by^\delta$ used in the all-at-once and reduced methods, respectively, and the corresponding operator value at the exact solution can be expressed as follows.
\[
\eqalign{
\|\bby^\delta-\bby\|_{\cW\times H\times\cY}= \|(0,0,\by^\delta)-(0,0,\bby)\|_{\cW\times H\times\cY}= 
\|\by^\delta-\by\|_{\cY}\\
=\|g(\cdot,\budag,\thdag)-g(\cdot,\bu^\delta,\thdag)-\bz^\delta\|_{\cY}
}
\]
where $(\budag,\thdag)$ and $(\bu^\delta,\thdag)$ solve \eqref{modex}, \eqref{obsex} and \eqref{mod}, \eqref{obs}, respectively.
Thus, assuming differentiability of $f$ and $g$ with respect to $u$ 
and boundedness of the derivatives by \eqref{growthfuftheta}, \eqref{growthgugtheta}, we can make use of the fact that the difference of the states satisfies
\[
\ddt(\budag-\bu^\delta)(t)={\hat{A}}(t)(\budag-\bu^\delta)(t)-\bw^\delta(t) \quad t\in(0,T)\,, \quad(\budag-\bu^\delta)(0)=0\,,
\]
with 
\[
{\hat{A}}(t)=\int_0^1 f_u'(t,\lambda\budag(t)+(1-\lambda)\bu^\delta(t),\thdag)\, d\lambda, \quad \|{\hat{A}}(t)\|_{V\to V^*}\leq c_{f,u}\,,
\]
with some constant $c_{f,u}$ (depending only on $\|\budag\|_\cU$, $\|\bu^\delta\|_\cU$), hence 
\[
(\budag-\bu^\delta)(t)=-\int_0^t \bw^\delta(\tau)\,d\tau
-\int_0^te^{\int_s^t{\hat{A}}(\sigma)\, d\sigma}{\hat{A}}(s) \int_0^s \bw^\delta(\tau)\,d\tau\, ds
\]
and therewith
\[
\|\budag-\bu^\delta\|_{L_2(0,T;V)}\leq c \|\int_0^\cdot \bw^\delta(\tau)\,d\tau\|_{L_2(0,T;V^*)}
\leq \tilde{c} \|\bw^\delta\|_{L_2(0,T;V^*)}\,.
\]
Thus altogether
\begin{equation}\label{delta}
\eqalign{
\|\bby^\delta-\bby\|_{\cW\times H\times\cY}=
\|\by^\delta-\by\|_{\cY}
\leq c_{g,u}\tilde{c}\delta_\bw+\delta_\bz=:\delta\,,
}
\end{equation}
where $\delta_\bw$, $\delta_\bz$ are bounds for the perturbations
\begin{equation}\label{deltawdeltaz}
\|\bw^\delta\|_{\cW}\leq\delta_\bw\,, \quad \|\bz^\delta\|_{\cY}\leq\delta_\bz
\end{equation}
and $c_{g,u}$ is the bound following from \eqref{growthgugtheta} (and depending only on $\|\budag\|_\cU$, $\|\bu^\delta\|_\cU$) in 
\[
\|\int_0^1 g_u'(t,\lambda\budag(t)+(1-\lambda)\bu^\delta(t),\thdag)\, d\lambda\|_{V\to V^*}\leq c_{g,u}\,.
\]

\subsection{Conditions on the operators}
Convergence conditions on the operators take the form of the 
\begin{itemize}
\item tangential cone condition in its all-at-once  
\begin{equation} \label{tangcone_aao}
\hspace*{-2.5cm}\eqalign{
\|f(\cdot,\tilde{\bu},\tilde{\theta})-f(\cdot,\bu,\theta)-f_u'(\cdot,\bu,\theta)(\tilde{\bu}-\bu)-f_\theta'(\cdot,\bu,\theta)(\tilde{\theta}-\theta)\|_\cW\nonumber\\
+\|u_0(\tilde{\theta})-u_0(\theta)-u_0'(\theta)(\tilde{\theta}-\theta)\|_H\nonumber\\
+\|g(\cdot,\tilde{\bu},\tilde{\theta})-g(\cdot,\bu,\theta)-g_u'(\cdot,\bu,\theta)(\tilde{\bu}-\bu)-g_\theta'(\cdot,\bu,\theta)(\tilde{\theta}-\theta)\|_\cY\nonumber\\
\leq 
c_{tc} \Bigl(\|\dot{\tilde{\bu}}-\dot{\bu}-f(\cdot,\tilde{\bu},\tilde{\theta})+f(\cdot,\bu,\theta)\|_\cW
+\|u_0(\tilde{\theta})-u_0(\theta)\|_H
+\|g(\cdot,\tilde{\bu},\tilde{\theta})-g(\cdot,\bu,\theta)\|_\cY
\Bigr)
}\hspace*{-2.5cm}
\end{equation}
and reduced version
\begin{equation} \label{tangcone_red}
\hspace*{-2.5cm}\|g(\cdot,\hat{\bu},\tilde{\theta})-g(\cdot,\bu,\theta)-g_u'(\cdot,\bu,\theta)\bv-g_\theta'(\cdot,\bu,\theta)(\tilde{\theta}-\theta)\|_\cY
\leq 
\tilde{c}_{tc} \|g(\cdot,\hat{\bu},\tilde{\theta})-g(\cdot,\bu,\theta)\|_\cY
\end{equation}
where $\hat{\bu}$, $\bu$, $\bv$ solve 
\begin{eqnarray}
\hspace*{-2cm}
\dot{\hat{\bu}}(t)=f(t,\hat{\bu}(t),\theta) \quad t\in(0,T)\,, \quad\hat{\bu}(0)=u_0(\tilde{\theta})\label{modutil}\\
\hspace*{-2cm}
\dot{\bu}(t)=f(t,\bu(t),\theta) \quad t\in(0,T)\,, \quad\bu(0)=u_0(\theta)
\label{modu}\\
\hspace*{-2cm}
\dot{\bv}(t)=f'_u(t,\bu(t),\theta)\bv(t)+f'_\theta(t,\bu(t),\theta)(\tilde{\theta}-\theta) \quad t\in(0,T)\,, \quad\bv(0)=u_0'(\theta)(\tilde{\theta}-\theta)\,,
\hspace*{-2cm}
\nonumber
\end{eqnarray}
\item the adjoint range invariance (which is sufficient for the tangential cone condition) in its all-at-once
\begin{equation}\label{adjrange_aao}
\hspace*{-2cm}\eqalign{
\ddt-f_u'(\cdot,\tilde{\bu},\tilde{\theta})=
\bR_{\cW\cW}^{\tilde{\bu},\tilde{\theta},\bu,\theta}(\ddt-f_u'(\cdot,\bu,\theta))
+\bR_{\cW H}^{\tilde{\bu},\tilde{\theta},\bu,\theta}\delta_0
+\bR_{\cW\cY}^{\tilde{\bu},\tilde{\theta},\bu,\theta}g_u'(\cdot,\bu,\theta)
\\
\delta_0=
\bR_{H\cW}^{\tilde{\bu},\tilde{\theta},\bu,\theta}(\ddt-f_u'(\cdot,\bu,\theta))
+\bR_{HH}^{\tilde{\bu},\tilde{\theta},\bu,\theta}\delta_0
+\bR_{H\cY}^{\tilde{\bu},\tilde{\theta},\bu,\theta}g_u'(\cdot,\bu,\theta)
\\
g_u'(\cdot,\tilde{\bu},\tilde{\theta})=
\bR_{\cY\cW}^{\tilde{\bu},\tilde{\theta},\bu,\theta}(\ddt-f_u'(\cdot,\bu,\theta))
+\bR_{\cY H}^{\tilde{\bu},\tilde{\theta},\bu,\theta}\delta_0
+\bR_{\cY\cY}^{\tilde{\bu},\tilde{\theta},\bu,\theta}g_u'(\cdot,\bu,\theta)
\\
-f_\theta'(\cdot,\tilde{\bu},\tilde{\theta})=
-\bR_{\cW\cW}^{\tilde{\bu},\tilde{\theta},\bu,\theta}f_\theta'(\cdot,\bu,\theta))
-\bR_{\cW H}^{\tilde{\bu},\tilde{\theta},\bu,\theta}u_0'(\theta)
+\bR_{\cW\cY}^{\tilde{\bu},\tilde{\theta},\bu,\theta}g_\theta'(\cdot,\bu,\theta)
\\
-u_0'(\tilde{\theta})=
-\bR_{H\cW}^{\tilde{\bu},\tilde{\theta},\bu,\theta}f_\theta'(\cdot,\bu,\theta))
-\bR_{HH}^{\tilde{\bu},\tilde{\theta},\bu,\theta}u_0'(\theta)
+\bR_{H\cY}^{\tilde{\bu},\tilde{\theta},\bu,\theta}g_\theta'(\cdot,\bu,\theta)
\\
g_\theta'(\cdot,\tilde{\bu},\tilde{\theta})=
-\bR_{\cY\cW}^{\tilde{\bu},\tilde{\theta},\bu,\theta}f_\theta'(\cdot,\bu,\theta))
-\bR_{\cY H}^{\tilde{\bu},\tilde{\theta},\bu,\theta}u_0'(\theta)
+\bR_{\cY\cY}^{\tilde{\bu},\tilde{\theta},\bu,\theta}g_\theta'(\cdot,\bu,\theta)
}
\end{equation}
and its reduced version
\begin{equation}\label{adjrange_red}
\hspace*{-2cm}
\eqalign{
g_u'(\cdot,\hat{\bu},\tilde{\theta})S'(\tilde{\theta})+g_\theta'(\cdot,\hat{\bu},\tilde{\theta})
=R_{\cY\cY}^{\tilde{\theta},\theta}(g_u'(\cdot,\bu,\theta)S'(\theta)+g_\theta'(\cdot,\bu,\theta))
\\
\mbox{where }\hat{\bu},\ \bu \mbox{ solve \eqref{modutil}, \eqref{modu} and $S'(\tilde{\theta}):\xi\mapsto \tilde{\bv}$, $S'(\theta):\xi\mapsto\bv$}
}
\end{equation}
where
\begin{equation}\label{vtilv}
\hspace*{-2cm}
\eqalign{
\dot{\tilde{\bv}}(t)=f'_u(t,\hat{\bu}(t),\tilde{\theta})\tilde{\bv}(t)+f'_\theta(t,\hat{\bu}(t),\theta)\xi \quad t\in(0,T)\,, \quad\tilde{\bv}(0)=u_0'(\tilde{\theta})\xi\\
\dot{\bv}(t)=f'_u(t,\bu(t),\theta)\bv(t)+f'_\theta(t,\bu(t),\theta)\xi \quad t\in(0,T)\,, \quad\bv(0)=u_0'(\theta)\xi}
\end{equation}
\item the range invariance in its all-at-once
\begin{equation}\label{range_aao}
\hspace*{-2cm}\eqalign{
\ddt-f_u'(\cdot,\tilde{\bu},\tilde{\theta})=
(\ddt-f_u'(\cdot,\bu,\theta))\bR_{\cU\cU}^{\tilde{\bu},\tilde{\theta},\bu,\theta}
-f_\theta'(\cdot,\bu,\theta)\bR_{\cX\cU}^{\tilde{\bu},\tilde{\theta},\bu,\theta}
\\
\delta_0=
\delta_0\bR_{\cU\cU}^{\tilde{\bu},\tilde{\theta},\bu,\theta}
-u_0'(\theta)\bR_{\cX\cU}^{\tilde{\bu},\tilde{\theta},\bu,\theta}
\\
g_u'(\cdot,\tilde{\bu},\tilde{\theta})=
g_u'(\cdot,\bu,\theta)\bR_{\cU\cU}^{\tilde{\bu},\tilde{\theta},\bu,\theta}
+g_\theta'(\cdot,\bu,\theta)\bR_{\cX\cU}^{\tilde{\bu},\tilde{\theta},\bu,\theta}
\\
-f_\theta'(\cdot,\tilde{\bu},\tilde{\theta})=
(\ddt-f_u'(\cdot,\bu,\theta))\bR_{\cU\cX}^{\tilde{\bu},\tilde{\theta},\bu,\theta}
-f_\theta'(\cdot,\bu,\theta)\bR_{\cX\cX}^{\tilde{\bu},\tilde{\theta},\bu,\theta}
\\
-u_0'(\tilde{\theta})=
\delta_0\bR_{\cU\cX}^{\tilde{\bu},\tilde{\theta},\bu,\theta}
-u_0'(\theta)\bR_{\cX\cX}^{\tilde{\bu},\tilde{\theta},\bu,\theta}
\\
g_\theta'(\cdot,\tilde{\bu},\tilde{\theta})=
g_u'(\cdot,\bu,\theta)\bR_{\cU\cX}^{\tilde{\bu},\tilde{\theta},\bu,\theta}
+g_\theta'(\cdot,\bu,\theta)\bR_{\cX\cX}^{\tilde{\bu},\tilde{\theta},\bu,\theta}
}
\end{equation}
and its reduced version
\begin{equation}\label{range_red}
\hspace*{-2cm}
\eqalign{
g_u'(\cdot,\hat{\bu},\tilde{\theta})S'(\tilde{\theta})+g_\theta'(\cdot,\hat{\bu},\tilde{\theta})
=(g_u'(\cdot,\bu,\theta)S'(\theta)+g_\theta'(\cdot,\bu,\theta))R_{\cX\cX}^{\tilde{\theta},\theta}\\
\mbox{where }\hat{\bu},\ \bu \mbox{ solve \eqref{modutil}, \eqref{modu} and $S'(\tilde{\theta}):\xi\mapsto \tilde{\bv}$, $S'(\theta):\xi\mapsto\bv$ with \eqref{vtilv}}\,.
}
\end{equation}
\end{itemize}
These conditions are supposed to hold locally, i.e., for all $\tilde{\bu},\bu$, $\tilde{\theta},\theta$ in a small neighborhood of $\budag$, $\thdag$, respectively and the linear operators $\bR_{NM}^{\tilde{\bu},\tilde{\theta},\bu,\theta}$, $R_{NM}^{\tilde{\theta},\theta}$, $N,M\in\{\cW,H,\cY,\cU,\cX\}$ in \eqref{adjrange_aao} -- \eqref{range_red} have to be sufficiently close (in operator norm) to the identity. The latter can be relaxed to closeness on the subspaces $\overline{\mathcal{R}(F'(\theta))}$ and $\mathcal{N}(F'(\theta))^\bot$ for \eqref{adjrange_red} and \eqref{range_red}, respectively.
The conditions are rephrased for the operators $\bF_j$, $F_j$ in \eqref{bFj}, \eqref{Fj} in a straightforward manner, by restriction to the subintervals $(\tau_j,\tau_{j+1})$.

All of these conditions obviously hold in the linear case, i.e., if for all $t\in(0,T)$, $f$, $g$, $u_0$ are linear functions of $(u,\theta)$
\[ 
f(t,u,\theta)=A(t)u+B(t)\theta\,, \quad u_0(\theta)=B_0\theta\,, \quad g(t,u,\theta)=C(t)u+E(t)\theta\,,
\]
hence also
\[
S:\theta\mapsto e^{\int_0^t A(\sigma)\, d\sigma}B_0
+\int_0^t B(\tau)\,d\tau
+\int_0^te^{\int_s^t A(\sigma)\, d\sigma} A(s) \int_0^s B(\tau)\,d\tau\, ds
\]
is linear.

\medskip

We now investigate a partially nonlinear case, where especially the model is still nonlinear, so that the differences between the all-at-once and the reduced versions of the methods under consideration get pronounced.
To verify conditions \eqref{adjrange_aao} -- \eqref{range_red} in the simplified but practically relevant setting, where the parameter $\theta$ enters the model equation only linearly, and where the observation operator is linear, 
\begin{equation}\label{AC}
\eqalign{
f(\cdot,\bu,\theta)=\hat{f}(\cdot,\bu)+B\theta\,,\\
g(\cdot,\bu,\theta)=C\bu+E\theta\,,\\
u_0(\theta)=B_0\theta
}
\end{equation}
with bounded linear operators
\begin{equation}\label{BB0C}
B:\cX\to \cW\,, \quad B_0:\cX\to H\,, \quad C:\cU\to\cY\,, \quad E:\cX\to\cY\,,
\end{equation}
\cite[Lemma 7]{aao16} can be employed and yields the following result. 

\begin{lemma} \label{lem:adjrange}
Let $f$ and $g$ in \eqref{mod}, \eqref{obs} be defined by \eqref{AC} with \eqref{BB0C}.
\begin{enumerate}
\item \label{lemadjrange_aao}
If $\hat{f}$ is G\^{a}teaux differentiable with respect to its second argument and 
\begin{equation}\label{fhatC}
\hspace*{-2cm}
\forall \bv\in \cU\, : \
\|\hat{f}_u'(\cdot,\tilde{\bu})\bv-\hat{f}_u'(\cdot,\bu)\bv\|_{\cW}
\leq c_R\|C\bv\|_\cY,
\end{equation}
then \eqref{adjrange_aao} holds with 
\begin{equation}\label{bR_adjrange_aao}
\bR^{\tilde{\bu},\tilde{\theta},\bu,\theta}=
\left(\begin{array}{ccc}
\mbox{id}&0&(\hat{f}_u'(\cdot,\bu)-\hat{f}_u'(\cdot,\tilde{\bu}))C^\dagger\\
0&\mbox{id}&0\\
0&0&\mbox{id}
\end{array}\right)
\end{equation}
and 
\[
\|\bR^{\tilde{\bu},\tilde{\theta},\bu,\theta}-\mbox{id}\|_{\cW\times H \times\cY\to\cW\times H \times\cY}\leq c_R\,.
\]
\item \label{lemadjrange_red}
If \eqref{fhatC} with $\tilde{\bu}=S(\tilde{\theta})$, $\bu=S(\theta)$ and additionally
\begin{equation}\label{Ainvbounded}
\forall \bv\in \cU\, : \ 
\|\bv\|_\cU\leq c_f\,\Bigl(\|(\ddt-\hat{f}_u'(\cdot,S(\tilde{\theta})))\bv\|_{\cU}+\|\bv(0)\|\Bigr)
\end{equation}
then \eqref{adjrange_red} holds with 
\begin{equation}\label{R_adjrange_red}
\hspace*{-2cm}
\eqalign{
R_{\cY\cY}^{\tilde{\theta},\theta}=
C\Bigl(\ddt-\hat{f}_u'(\cdot,S(\tilde{\theta})),\delta_0\Bigr)^{-1}(\ddt-\hat{f}_u'(\cdot,S(\theta)),\delta_0)C^\dagger\\
=\mbox{Proj}_{\overline{\mathcal{R}(C)}}
+C\Bigl(\ddt-\hat{f}_u'(\cdot,S(\tilde{\theta})),\delta_0\Bigr)^{-1}\Bigl((\hat{f}_u'(\cdot,S(\tilde{\theta}))-\hat{f}_u'(\cdot,S(\theta))C^\dagger,0\Bigr)
}
\end{equation}
and 
\[
\|(R_{\cY\cY}^{\tilde{\theta},\theta}-\mbox{id})\mbox{Proj}_{\overline{\mathcal{R}(F'(\theta))}}\|_{\cY\to\cY}\leq c_Rc_f\|C\|_{\cU\to\cY}=:\tilde{c}_R
\]
\end{enumerate}
Here $C^\dagger$ is the Moore-Penrose generalized inverse of $C:\cU\to\cY$.
\end{lemma} 
\begin{proof}
For the sake of completeness of exposition and due to its shortness, the proof will be carried directly, without referring to \cite[Lemma 7]{aao16}.
Due to the fact that \eqref{fhatC} implies 
\[
\mathcal{N}(C)\subseteq\mathcal{N}(\hat{f}_u'(\cdot,\tilde{\bu})-\hat{f}_u'(\cdot,\bu)),
\]
hence 
\[
\hspace*{-2cm}(\hat{f}_u'(\cdot,\tilde{\bu})-\hat{f}_u'(\cdot,\bu))C^\dagger C
=(\hat{f}_u'(\cdot,\tilde{\bu})-\hat{f}_u'(\cdot,\bu))\mbox{Proj}^\cU_{\mathcal{N}(C)^\bot}
=(\hat{f}_u'(\cdot,\tilde{\bu})-\hat{f}_u'(\cdot,\bu)),
\]
formal verification of \eqref{adjrange_aao} and \eqref{adjrange_red} with \eqref{bR_adjrange_aao} and \eqref{R_adjrange_red}, respectively, in the setting \eqref{AC} is straightforward:
\[
\hspace*{-2cm}\eqalign{
\ddt-\hat{f}_u'(\cdot,\tilde{\bu})=
\mbox{id}(\ddt-\hat{f}_u'(\cdot,\bu))
+0 \delta_0
+(\hat{f}_u'(\cdot,\bu)-\hat{f}_u'(\cdot,\tilde{\bu}))C^\dagger C
\\
\delta_0=
0 (\ddt-\hat{f}_u'(\cdot,\bu))
+\mbox{id}\delta_0
+0 C
\\
C=
0 (\ddt-\hat{f}_u'(\cdot,\bu))
+0 \delta_0
+\mbox{id} C
\\
-B=
-\mbox{id} B 
-0 B_0
+ (\hat{f}_u'(\cdot,\bu)-\hat{f}_u'(\cdot,\tilde{\bu}))C^\dagger 0
\\
-B_0=
-0 B
-\mbox{id} B_0
+ 0
\\
E=
- 0 B
- 0 B_0
+ \mbox{id}E
}
\]
and, since in the reduced case $S'(\theta)=-(\ddt-\hat{f}_u'(\cdot,S(\theta)),\delta_0)^{-1}B$ holds,
\[
\hspace*{-2cm}
\eqalign{CS'(\tilde{\theta})+0
=C(\ddt-\hat{f}_u'(\cdot,S(\tilde{\theta})),\delta_0)^{-1}(\ddt-\hat{f}_u'(\cdot,S(\theta)),\delta_0)C^\dagger
(CS'(\theta)+0)\,.
}
\]
The difference from the identity on the relevant subspaces can be seen as follows:
\[
\eqalign{
\|\bR^{\tilde{\bu},\tilde{\theta},\bu,\theta}-\mbox{id}\|_{\cW\times H \times\cY\to\cW\times H \times\cY}=\|(\hat{f}_u'(\cdot,\bu)-\hat{f}_u'(\cdot,\tilde{\bu}))C^\dagger\|_{\cY\to\cW}
\\
=\sup_{\bv\in\mathcal{N}(C)^\bot\setminus\{0\}}
\frac{\|\hat{f}_u'(\cdot,\bu)-\hat{f}_u'(\cdot,\tilde{\bu})\bv\|_\cW}{\|C\bv\|_\cY}
\leq c_R\,,
}
\]
and, since in the reduced case we additionally have $\mathcal{R}(F'(\theta))=\mathcal{R}(C)$,
\[
\eqalign{
\|(R_{\cY\cY}^{\tilde{\theta},\theta}-\mbox{id})\mbox{Proj}_{\overline{\mathcal{R}(F'(\theta))}}\|_{\cY\to\cY}\\
=\|C\Bigl(\ddt-\hat{f}_u'(\cdot,S(\tilde{\theta})),\delta_0\Bigr)^{-1}\Bigl(\hat{f}_u'(\cdot,S(\tilde{\theta}))-\hat{f}_u'(\cdot,S(\theta)),0\Bigr)C^\dagger\|_{\cY\to\cY}
\\
\leq c_f \|C\|_{\cU\to\cY}
\|(\hat{f}_u'(\cdot,S(\theta))-\hat{f}_u'(\cdot,S(\tilde{\theta})))C^\dagger\|_{\cY\to\cW}
\leq c_R c_f \|C\|_{\cU\to\cY}
}
\]
\end{proof}

For the alternative range invariance condition, which can be used to prove convergence of the IRGNM type methods, analogously we get the following result.

\begin{lemma} \label{lem:range}
Let $f$ and $g$ in \eqref{mod}, \eqref{obs} be defined by \eqref{AC} with \eqref{BB0C}.
\begin{enumerate}
\item \label{lemrange_aao}
If $\hat{f}$ is G\^{a}teaux differentiable with respect to its second argument and 
\begin{equation}\label{fhatB}
\hspace*{-2cm}
\forall \bw\in \cW\, : \
\|\hat{f}_u'(\cdot,\tilde{\bu})^*I\bw-\hat{f}_u'(\cdot,\bu)^*I\bw\|_\cW
\leq c_R\|B^*\bw\|_\cY,
\end{equation}
are satisfied, then \eqref{range_aao} holds with 
\begin{equation}\label{bR_range_aao}
\bR^{\tilde{\bu},\tilde{\theta},\bu,\theta}=
\left(\begin{array}{cc}
\mbox{id}&0\\
B^\dagger(\hat{f}_u'(\cdot,\tilde{\bu})-\hat{f}_u'(\cdot,\bu))&\mbox{id}
\end{array}\right)
\end{equation}
and 
\[
\|\bR^{\tilde{\bu},\tilde{\theta},\bu,\theta}-\mbox{id}\|_{\cU\times\cX\to\cU\times\cX}\leq c_R\,.
\]
\item \label{lemrange_red}
If \eqref{fhatB} with $\tilde{\bu}=S(\tilde{\theta})$, $\bu=S(\theta)$ and additionally \eqref{Ainvbounded},
then \eqref{range_red} holds with 
\begin{equation}\label{R_range_red}
\hspace*{-2cm}
\eqalign{
R_{\cX\cX}^{\tilde{\theta},\theta}\xi=
B^\dagger(\ddt-\hat{f}_u'(\cdot,S(\theta)))\Bigl(\ddt-\hat{f}_u'(\cdot,S(\tilde{\theta})),\delta_0\Bigr)^{-1}(B\xi,0)\\
=\mbox{Proj}_{\mathcal{N}(B)^\bot}\xi
+B^\dagger\Bigl(\hat{f}_u'(\cdot,S(\tilde{\theta}))-\hat{f}_u'(\cdot,S(\theta))\Bigr)\Bigl(\ddt-\hat{f}_u'(\cdot,S(\tilde{\theta}))B,\delta_0\Bigr)^{-1}(B\xi,0)
}
\end{equation}
\[
\|\mbox{Proj}_{\mathcal{N}(F'(\theta))^\bot}(R_{\cX\cX}^{\tilde{\theta},\theta}-\mbox{id})\|_{\cX\to\cX}\leq c_Rc_f\|B\|_{\cX\to\cY}=:\tilde{c}_R
\]
\end{enumerate}
\end{lemma} 

\subsection{Convergence}
As a consequence we can deduce the following convergence results
\begin{corollary}\label{cor:conv}
Let $f$ and $g$ in \eqref{mod}, \eqref{obs} be defined by \eqref{AC} with \eqref{BB0C}.
\begin{enumerate}
\item 
If the assumptions of Lemma \ref{lem:adjrange} \eqref{lemadjrange_aao} are satisfied with $c_R<\frac13$ in a neighborhood of some solution $(\budag,\thdag)$ of \eqref{modex}, \eqref{obsex}, and $\hat{f}$ satisfying \eqref{growthf}, \eqref{growthfuftheta},
then the aLW, the aLWK, the aIRGNM as defined in Subsections \ref{sec:aaoLW}, \ref{sec:aaoLWK}, \ref{sec:aaoIRGNM}, respectively, with the stopping index $k_*$ chosen by the discrepancy principle
\begin{equation}\label{discrprinc_aao}
k_*(\delta)=\min\{k\in\mathbb{N}_0\ : \ \|\bF(\bu_k,\theta_k)-\bby^\delta\|_{\cW\times H\times\cY}\leq\tau\delta\}
\end{equation}
with $\delta$ as in \eqref{delta}, \eqref{deltawdeltaz} and $\tau$ sufficiently large but fixed, are well-defined and converge to $(\budag,\thdag)$ in the sense of a regularization method
\[
\|(\bu_{k_*(\delta)},\theta_{k_*(\delta)})-(\budag,\thdag)\|_{\cU\times\cX}\to0
\mbox{  as }\delta\to0
\]
locally, i.e., provided the starting point $(\bu_0,\theta_0)$ is sufficiently close in norm to $(\budag,\thdag)$.
\item 
If the assumptions of Lemma \ref{lem:adjrange} \eqref{lemadjrange_red} are satisfied with $\tilde{c}_R<\frac13$ in a neighborhood of some solution $(\budag,\thdag)$ of \eqref{modex}, \eqref{obsex}, with the parameter-to-state-map $S$ well-defined by \eqref{S} (cf., e.g., Proposition \ref{prop:Swelldef}), and $\hat{f}$ satisfying \eqref{growthf}, \eqref{fubounded}, \eqref{fthetabounded}, 
then the rLW, the rLWK, the rIRGNM as defined in Subsections \ref{sec:redLW}, \ref{sec:redLWK}, \ref{sec:redIRGNM}, respectively, with the stopping index $k_*$ chosen by the discrepancy principle
\[
k_*(\delta)=\min\{k\in\mathbb{N}_0\ : \ \|F(\bu_k,\theta_k)-\by^\delta\|_\cY\leq\tau\delta\}
\]
with $\delta$ as in \eqref{delta}, \eqref{deltawdeltaz} and $\tau$ sufficiently large but fixed, 
are well-defined and converge to $\thdag$ in the sense of a regularization method
\[
\|\theta_{k_*(\delta)}-\thdag\|_{\cX}\to0
\mbox{  as }\delta\to0
\]
locally, i.e., provided the starting point $\theta_0$ is sufficiently close in norm to $\thdag$.
\item 
If the assumptions of Lemma \ref{lem:range} \eqref{lemrange_aao} are satisfied with $c_R$ sufficiently small in a neighborhood of some solution $(\budag,\thdag)$ of \eqref{modex}, \eqref{obsex}, and $\hat{f}$ satisfying \eqref{growthf}, \eqref{growthfuftheta},
then the aIRGNM as defined in Subsection \ref{sec:aaoIRGNM} with a stopping index $k_*$ chosen such that 
\begin{equation}\label{kstapriori}
k_*(\delta)\to\infty \mbox{ and }\frac{\delta^2}{\alpha_{k_*(\delta)}}\to0 \mbox{ as }\delta\to0
\end{equation}
with $\delta$ as in \eqref{delta}, \eqref{deltawdeltaz}, 
is well-defined and converges to $(\budag,\thdag)$ in the sense of a regularization method
\[
\|(\bu_{k_*(\delta)},\theta_{k_*(\delta)})-(\budag,\thdag)\|_{\cU\times\cX}\to0
\mbox{  as }\delta\to0
\]
locally, i.e., provided the starting point $(\bu_0,\theta_0)$ is sufficiently close in norm to $(\budag,\thdag)$.
\item 
If the assumptions of Lemma \ref{lem:range} \eqref{lemrange_red} are satisfied with $\tilde{c}_R$ sufficiently small in a neighborhood of some solution $(\budag,\thdag)$ of \eqref{modex}, \eqref{obsex}, with the parameter-to-state-map $S$ well-defined by \eqref{S} (cf., e.g., Proposition \ref{prop:Swelldef}), and $\hat{f}$ satisfying \eqref{growthf}, \eqref{fubounded}, \eqref{fthetabounded}, 
then the rIRGNM as defined in Subsection \ref{sec:redIRGNM} with a stopping index $k_*$ chosen according to \eqref{kstapriori}, with $\delta$ as in \eqref{delta}, \eqref{deltawdeltaz}, 
is well-defined and converges to $\thdag$ in the sense of a regularization method
\[
\|\theta_{k_*(\delta)}-\thdag\|_{\cX}\to0
\mbox{  as }\delta\to0
\]
locally, i.e., provided the starting point $\theta_0$ is sufficiently close in norm to $\thdag$.
\end{enumerate}
\end{corollary}
\begin{proof}
The assertions follow from Theorems 2.6, 3.26, 4.12, 4.13 in \cite{KNS08} (see also the original references \cite{BNS97,HNS95,KS02}), using the fact that the adjoint range invariance conditions \eqref{adjrange_aao}, \eqref{adjrange_red} imply the tangential cone conditions \eqref{tangcone_aao},  \eqref{tangcone_red} with $c_{tc}=\frac{c_R}{1-c_R}$, $\tilde{c}_{tc}=\frac{\tilde{c}_R}{1-\tilde{c}_R}$, respectively. Therewith ``$\tau$ sufficiently large'' can be specified to 
$\tau\geq2\frac{1+c_{tc}}{1-2c_{tc}}=\frac{2}{1-3c_R}$
for the Landweber type methods (for the IRGNM type methods, the expression is more complicated and thus skipped here). Moreover, for the respective Kaczmarz versions, it is readily checked that under the Assumptions of Lemma \ref{lem:adjrange} \eqref{lemadjrange_aao}, \eqref{lemadjrange_red} its conclusions remain valid also for the operators $\bF_j$, $F_j$ from \eqref{bFj}, \eqref{Fj}. 
Note that for these assertions to hold, $\bF'$ and $F'$ need not necessarily be Fr\'{e}chet differentiable; G\^ateaux differentiability with linear and continuous derivatives is enough for our purposes. 
\end{proof}

\subsection{An Example}\label{sec:ex}
The convergence results of Corollary \ref{cor:conv} apply, e.g., to the problem of identifying the stationary source $\theta=(\theta_\omega,\theta_\Gamma)$ in a semilinear diffusion system from linear observations
\begin{eqnarray*}
\dot{\bu}(t)=\Delta\bu(t)-\Phi(\bu(t))+\chi_\omega(\theta_\omega)+\bw^\delta(t) \quad t\in(0,T)\times\Omega, \\
\frac{\partial \bu}{\partial \nu}=\theta_\Gamma\quad t\in(0,T)\times\Gamma,\\
\bu=0\quad t\in(0,T)\times(\partial\Omega\setminus\Gamma),\\
\bu(0)=u_0,\\
\by^\delta(t)=C\bu(t)+\bz^\delta(t)\quad t\in(0,T),
\end{eqnarray*}
with $\Omega\subseteq\R^d$ a bounded Lipschitz domain. 
The source $\theta=(\theta_\omega,\theta_\Gamma)\in L^2(\omega)\times L^2(\Gamma)$ consists of a boundary part $\theta_\Gamma$ an interior part $\theta_\omega$ acting on a subdomain $\omega\subseteq\Omega$ and extended by zero to $\Omega$
\[
\chi_\omega: L_2(\omega)\to C(0,T;L_2(\Omega))\,, \quad
(\chi_\omega\phi)(x,t)=\left\{\begin{array}{ll} \phi(x)\mbox{ for }x\in \omega\\
0\mbox{ else.}\end{array}\right.
\]
We aim at verifying the conditions for the all-at-once versions. To this end, we assume the function $\Phi:\R\to\R$ to be differentiable and satisfy the growth conditions
\begin{equation}\label{growthPhi}
|\Phi(\lambda)|\leq c_\Phi(1+|\lambda|^\gamma)\,, \quad 
|\Phi'(\lambda)|\leq c_{\Phi'}(1+|\lambda|^{\gamma'})
\end{equation}
with 
\begin{eqnarray}\label{gamma}
&&\gamma\leq 3 \mbox{ if }d=1, \quad \gamma<3\mbox{ if }d=2, \quad \gamma
\leq 1+\frac{4}{d}\mbox{ if }d\geq3,\\
&&\gamma'\leq2  \mbox{ if }d=1, \quad \gamma'<2\mbox{ if }d=2, \quad \gamma'\leq \frac{4}{3}\mbox{ if }d=3\quad \gamma'\leq \frac{2}{d-2}\mbox{ if }d\geq4.
\nonumber
\end{eqnarray}
No monotonicity is imposed on $\Phi$ since we do not need a parameter-to-state map.

With the function spaces
\[
\hspace*{-2cm}
\eqalign{
\cX= L^2(\omega)\times L^2(\Gamma)\\
V=H_\Gamma^1(\Omega)=\{v\in H^1(\Omega)\ : \ \mbox{tr}_{\partial\Omega\setminus\Gamma}v=0 \}\supseteq H_0^1(\Omega)\mbox{ with norm } \|v\|_V=\|\nabla v\|_{L^2(\Omega)}\\
H=L^2(\Omega)\,,
}
\]
the initial data $u_0\in H=L_2(\Omega)$, an observation operator $C$ mapping $V$ linearly and continuously into some observation space $Z$, e.g.,
\[
Cv=v|_{\Sigma}\,, \quad Z=L_2(\Sigma) 
\]
for $\Sigma$ an open subset of $\Omega$ or part of its boundary $\partial\Omega$, (in which case $C$ is a trace operator),
perturbations $\bw^\delta\in L_2(0,T;V^*)\subseteq L_2(0,T;H^{-1}(\Omega))$, $\bz^\delta\in L_2(0,T;Z)$,
and the variational formulation of the model
\[
\hspace*{-2cm}\eqalign{
\forall^{\mbox{\footnotesize{a.e.}}}t\in(0,T)\,,\forall v\in V\ : \ \int_\Omega \dot{\bu}(t)v\,dx
=-\int_\Omega \Bigl(\nabla\bu(t)\cdot\nabla v+\Phi(\bu(t))v +\bw^\delta(t)v\Bigr)\, dx\\
\hspace*{4.5cm}+\int_\omega \theta_\omega v\, dx +\int_\Gamma \theta_\Gamma v\, ds\\
\bu(0)=u_0
}
\]
this fits into the above framework. Indeed, the growth condition \eqref{growthPhi} implies that $\hat{f}$ defined by 
\[
\forall u,v\in V\ : \ \dup{\hat{f}(u)}{v}{V}=-\int_\Omega \Bigl(\nabla u\cdot\nabla v+\Phi(u)v \Bigr)\, dx
\]
satisfies the growth estimate
\begin{equation}\label{growthestPhi}
\hspace*{-2cm}
\eqalign{
\|\hat{f}(u)\|_{V^*}
=\sup_{v\in V, \|\nabla v\|_{L_2(\Omega)}\le1} \left|\int_\Omega \Bigl(\nabla u\cdot\nabla v+\Phi(u)v \Bigr)\, dx\right|\\
\leq\|u\|_V+c_\Phi\sup_{v\in V, \|\nabla v\|_{L_2(\Omega)}\le1} \int_\Omega(1+|u|^\gamma)|v|\, dx\\
\leq\|u\|_V+c_\Phi c_{PF}\sqrt{|\Omega|}+c_\Phi\sup_{v\in V, \|\nabla v\|_{L_2(\Omega)}\le1} \int_\Omega
|u|^{\gamma-1}|uv|\, dx\\
\leq\|u\|_V+c_\Phi c_{PF}\sqrt{|\Omega|}+c_\Phi\sup_{v\in V, \|\nabla v\|_{L_2(\Omega)}\le1} 
\|u\|_{L_2(\Omega)}^{\gamma-1}\|u\|_{L_\frac{4}{3-\gamma}(\Omega)}
\|v\|_{L_\frac{4}{3-\gamma}(\Omega)}\,,
}
\end{equation}
for any $v\in V$, where we have used H\"older's inequality with $p=\frac{2}{\gamma-1}$, $p^*=\frac{2}{3-\gamma}$, and the Cauchy-Schwarz inequality, and $c_{PF}$ is the constant in the Poincar\'{e}-Friedrichs inequality
\[
\|v\|_{L_2(\Omega)}\leq c_{PF}\|\nabla v\|_{L_2(\Omega)}\,.
\]
Since boundedness of the embedding 
\begin{equation}\label{pbar}
V\subseteq H^1(\Omega)\to L_{\bar{p}}(\Omega) \mbox{ with } 
\bar{p}\left\{\begin{array}{ll}
=\infty&\mbox{ if }d=1,\\ 
<\infty&\mbox{ if }d=2,\\ 
=\frac{2d}{d-2}&\mbox{ if }d\geq3
\end{array}\right.
\end{equation}
with constant $c_{V\to L_{\bar{p}}(\Omega)}$, and $\frac{4}{3-\gamma}\leq\bar{p}$ holds under constraint \eqref{gamma}, estimate \eqref{growthestPhi} implies \eqref{growthf} with $\psi(\lambda)=c(1+\lambda^{\gamma-1})$ for some constant $c$.

The growth condition for differentiability also follows from boundedness of the embedding \eqref{pbar} (where we choose $\bar{p}>4$ and $\bar{p}\geq\frac{4}{2-\gamma'}$ in case $d=2$),
together with the following estimates.
H\"older's inequality with 
$p=\frac{\bar{p}}{2}$,
$p^*=\frac{\bar{p}}{\bar{p}-2}$
yields
\[
\hspace*{-2cm}
\eqalign{
\|\Phi'(v)\|_{L^2(\Omega)\to V^*}
=\sup_{z\in L_2(\Omega), \|z\|_{L_2}\leq 1}\|\Phi'(v)z\|_{V^*}
=\sup_{\stackrel{z\in L_2(\Omega), \|z\|_{L_2}\leq 1}{w\in V, \|w\|_{V}\leq 1}}
\int_\Omega\Phi'(v)\,z\,w\, dx\\
=\sup_{w\in V, \|w\|_{V}\leq 1} \|\Phi'(v)\,w\|_{L_2(\Omega)}
\leq c_{V\to L_{\bar{p}}} \|\Phi'(v)\|_{L_{\frac{2\bar{p}}{\bar{p}-2}}(\Omega)}
}
\]
In view of the growth condition \eqref{growthPhi} it therefore remains to estimate the expression
\begin{equation}\label{est:v4L1}
\||v|^{\gamma'}\|_{L_{\frac{2\bar{p}}{\bar{p}-2}}(\Omega)}
=\|v^{\frac{2\gamma'\bar{p}}{\bar{p}-2}}\|_{L_1(\Omega)}^{\frac{\bar{p}-2}{2\bar{p}}}\,.
\end{equation}
In case $d\in\{2,3\}$, 
this can be done by using H\"older's inequality with 
$p=\frac{\bar{p}-2}{2}>1$, $p^*=\frac{\bar{p}-2}{\bar{p}-4}$ so that redefining $\gamma'=\max\{1,\gamma'\}$ we get, in both cases $d=2,3$, by \eqref{gamma}, 
$
\frac{2(\gamma'-1)\bar{p}}{\bar{p}-2}p^*=\frac{2(\gamma'-1)\bar{p}}{\bar{p}-4}=:q\leq2
$, 
hence 
\[
\hspace*{-2cm}\eqalign{
\|v^{\frac{2\gamma'\bar{p}}{\bar{p}-2}}\|_{L_1(\Omega)}^{\frac{\bar{p}-2}{2\bar{p}}}
=\||v|^{\frac{2(\gamma'-1)\bar{p}}{\bar{p}-2}} |v|^{\frac{2\bar{p}}{\bar{p}-2}}\|_{L_1(\Omega)}^{\frac{\bar{p}-2}{2\bar{p}}}
\leq |\Omega|^{\frac{(\gamma'-1)(2-q)}{2}} c_{V\to L_{\bar{p}}} \|v\|_{L_2(\Omega)}^{\gamma'-1}\|v\|_{V}\,,
}
\]
i.e., \eqref{growthfuftheta} with $\phi_{H,u}^2(\lambda)=|\Omega|^{\frac{(\gamma'-1)(2-q)}{2}} c_{V\to L_{\bar{p}}}|\lambda|^{\gamma'-1}$.\\
In case $d=1$, where $\bar{p}=\infty$, $\frac{2\bar{p}}{\bar{p}-2}=2$, this estimate (with $q=\frac{3\gamma'-4}{\gamma'-1}$) can be easily verified without using H\"older's inequality.\\
For $d\geq4$ we have 
$
\frac{2\gamma'\bar{p}}{\bar{p}-2}=:q\leq \bar{p} \mbox{ and } \gamma'\leq1
$,
hence
\[ 
\|v^{\frac{2\gamma'\bar{p}}{\bar{p}-2}}\|_{L_1(\Omega)}^{\frac{\bar{p}-2}{2\bar{p}}}
=\|v\|_{L_{\frac{2\gamma'\bar{p}}{\bar{p}-2}}(\Omega)}^{\gamma'}
\leq |\Omega|^{\frac{\bar{p}-q}{\bar{p}}} \|v\|_{L_{\bar{p}}(\Omega)}^{\gamma'}
\leq |\Omega|^{\frac{\bar{p}-q}{\bar{p}}}  c_{V\to L_{\bar{p}}}^{\gamma'} (1+\|v\|_{V}).
\]

\medskip

The adjoint range invariance condition \eqref{adjrange_aao} can be verified in case of full observations, i.e.,
\begin{equation}\label{Cfull}
C:\cU\to\cY=L_2(0,T;L_2(\Omega))\,, \quad v\mapsto v\,,
\end{equation}
provided additionally $\Phi'$ is H\"older continuous with 
\begin{equation}\label{growthPhipp}
|\Phi'(\lambda)-\Phi'(\tilde{\lambda})|\leq c_{\Phi''}(1+|\lambda|^{\gamma''}+|\tilde{\lambda}|^{\gamma''})|\tilde{\lambda}-\lambda|^\kappa
\end{equation}
and 
\begin{equation}\label{gammapp}
\gamma''+\kappa\leq \frac{\bar{p}-2}{\bar{p}}
\end{equation}
for $\bar{p}$ as in \eqref{pbar}.
Namely the sufficient condition \eqref{fhatC} for adjoint range invariance can be verified,  
using H\"older's inequality with $p=\frac{2}{\bar{p}^*}$, $p^*=\frac{2}{2-\bar{p}^*}$, $p^*\bar{p}^*=\frac{2\bar{p}^*}{2-\bar{p}^*}$
\[
\hspace*{-2cm}\eqalign{
\|
\hat{f}_u'(\cdot,\tilde{\bu})\bv-\hat{f}_u'(\cdot,\bu)\bv\|_\cW
\leq c_{V\to L_{\bar{p}}(\Omega)} \left(\int_0^T \|(\Phi'(\tilde{\bu}(t))-\Phi'(\bu(t)))\bv(t)\|_{L_{\bar{p}^*}(\Omega)}^2\, dt \right)^{1/2}
\\
\leq c_{V\to L_{\bar{p}}(\Omega)} c_{\Phi''}\left(\int_0^T \|(1+|\tilde{\bu}(t)|^{\gamma''}+|\bu(t)|^{\gamma''}) |\tilde{\bu}(t)-\bu(t)|^\kappa\|_{L_{\frac{2\bar{p}^*}{2-\bar{p}^*}}(\Omega)}^2 \|\bv(t)\|_{L_2(\Omega)}^2
\, dt\right)^{1/2}\\
\leq c_{V\to L_{\bar{p}}(\Omega)} c_{\Phi''}\|(1+|\tilde{\bu}|^{\gamma''}+|\bu|^{\gamma''}) |\tilde{\bu}-\bu|^\kappa\|_{C(0,T;L_{\frac{2\bar{p}^*}{2-\bar{p}^*}}(\Omega))}^2 \|\bv\|_\cY
}
\]
where $(\kappa+\gamma'')\frac{2\bar{p}^*}{2-\bar{p}^*}=(\kappa+\gamma'')\frac{2\bar{p}}{\bar{p}-2}\leq 2$, so that we can again rely on boundedness of the embedding $\cU\subseteq C(0,T;L_2(\Omega))$.

Unless $\Phi$ is linear, the range invariance condition \eqref{range_aao} cannot be expected to hold in the setting with a time independent parameter, since in case of full observations it can only hold if $\bR^{\tilde{\bu},\bu,\tilde{\theta},\theta}_{\cU,\cU}=\mbox{id}$ (note that $E=0$), hence the first line in \eqref{range_aao} would imply 
\[
\forall \bv\in\cU\, \forall^{\mbox{\footnotesize{a.e.}}}t\in(0,T)\, : \  (\Phi'(\bu(t)-\Phi'(\tilde{\bu}(t))\bv(t)=(\chi_\omega \bR^{\tilde{\bu},\bu,\tilde{\theta},\theta}_{\cX,\cU} \bv)(t)\,,
\]
so that by definition of $\chi_\omega$, the difference $(\Phi'(\bu)-\Phi'(\tilde{\bu}))\bv$ would have to be constant in time for all $\bv\in\cU$, which can only hold if this difference vanishes. However, for the latter to hold for all $\tilde{\bu}$ in a $\cU$ neighborhood of $\bu$, $\Phi$ would necessarily have to be linear.

\begin{corollary}
Let \eqref{growthPhi}, \eqref{gamma}, \eqref{Cfull}, \eqref{growthPhipp}, \eqref{gammapp} be satisfied.
\\
Then there exists $\rho>0$ such that for all $\bu_0\in L^2(\Omega)$, $\theta\in L^2(\omega)\times L^2(\Gamma)$ satisfying $\|\bu_0-\budag\|_{L^2(\Omega)}+\|\theta_0-\thdag\|_{L^2(\omega)\times L^2(\Gamma)}<\rho$, the aLW, the aLWK, the aIRGNM as defined in Subsections \ref{sec:aaoLW}, \ref{sec:aaoLWK}, \ref{sec:aaoIRGNM}, respectively, with the stopping index $k_*$ chosen by the discrepancy principle \eqref{discrprinc_aao} with $\delta$ as in \eqref{delta}, \eqref{deltawdeltaz} and $\tau$ sufficiently large but fixed, are well-defined and converge to $(\budag,\thdag)$ in the sense of a regularization method
\[
\hspace*{-2cm}
\|\bu_{k_*(\delta)}-\budag\|_{L^2(0,T;H_\Gamma^1(\Omega))}+
\|\dot{\bu}_{k_*(\delta)}-\dot{\bu}^\dagger\|_{L^2(0,T;(H_\Gamma^1(\Omega))^*)}+
\|\theta_{k_*(\delta)}-\thdag\|_{L^2(\omega)\times L^2(\Gamma)}\to0
\]
as $\delta\to0$.
\end{corollary}
In Figures \ref{fig:redLW}, \ref{fig:aaoLW} we provide a very preliminary and exemplary comparison of the reduced and the all-at-once Landweber iteration for the above example with
\[
\eqalign{
\Omega=\omega=(0,1), \ T=\frac{1}{10},\ \thdag(x)=\frac{\sin(2\pi x)}{10}, \ u_0\equiv0, \ \Phi(\lambda)=10\,\mbox{sign}(\lambda)\lambda^2, \\
\theta_0\equiv0, \ \bu_0\equiv0, \ \mu_k\equiv 1.
}
\]
Figures \ref{fig:redLW}, \ref{fig:aaoLW} show the results after 50000 iterations of the reduced and the all-at-once Landweber method from Subsections \ref{sec:redLW} and \ref{sec:aaoLW}, respectively, where in both cases we used imlicit Euler timestepping with a constant step size of $10^{-3}$ and took advantage of a diagonalization of the discretized (with spatial step size $10^{-2}$) Lapace operator. The results are very similar, whereas the cpu times differ by a factor of more than five in favor of the all-at-once (379.56 sec) versus the reduced (1996.88 sec) method. We expect this difference to get even more pronounced when exploiting the possibility of using fast integrators for evaluation the variation of constants formulas \eqref{varofconst}.
\begin{figure}
\includegraphics[width=0.33\textwidth]{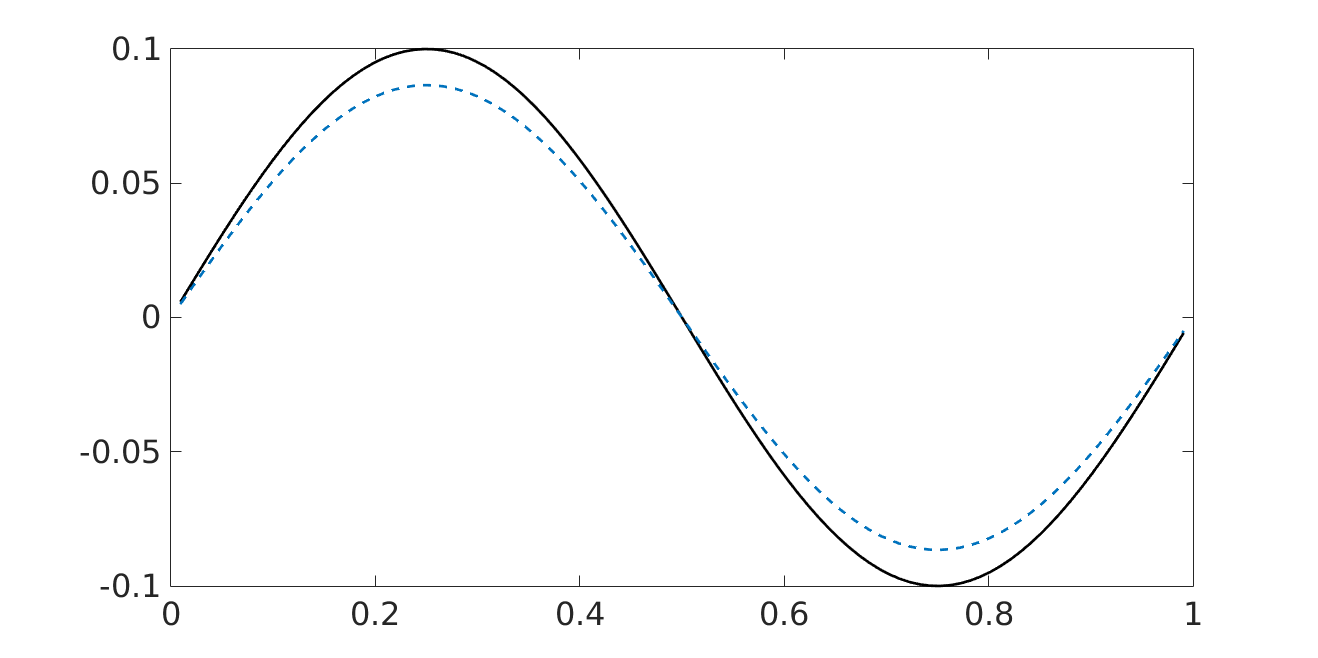}
\includegraphics[width=0.33\textwidth]{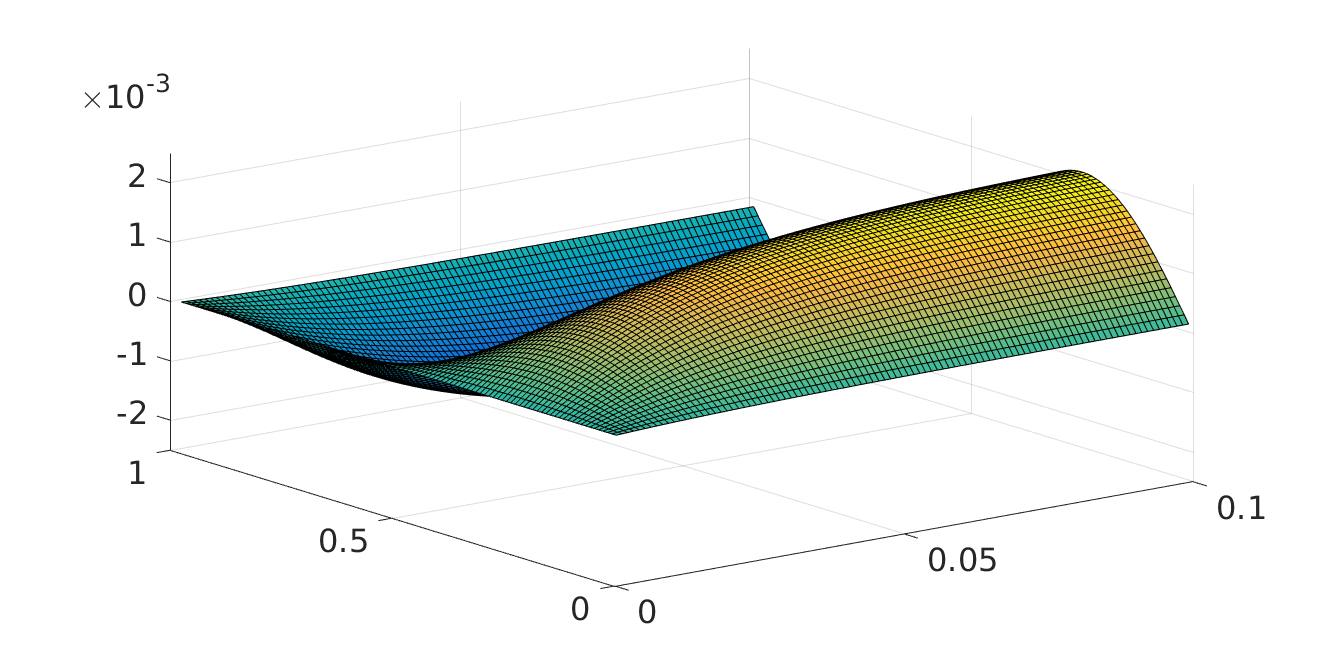}
\includegraphics[width=0.33\textwidth]{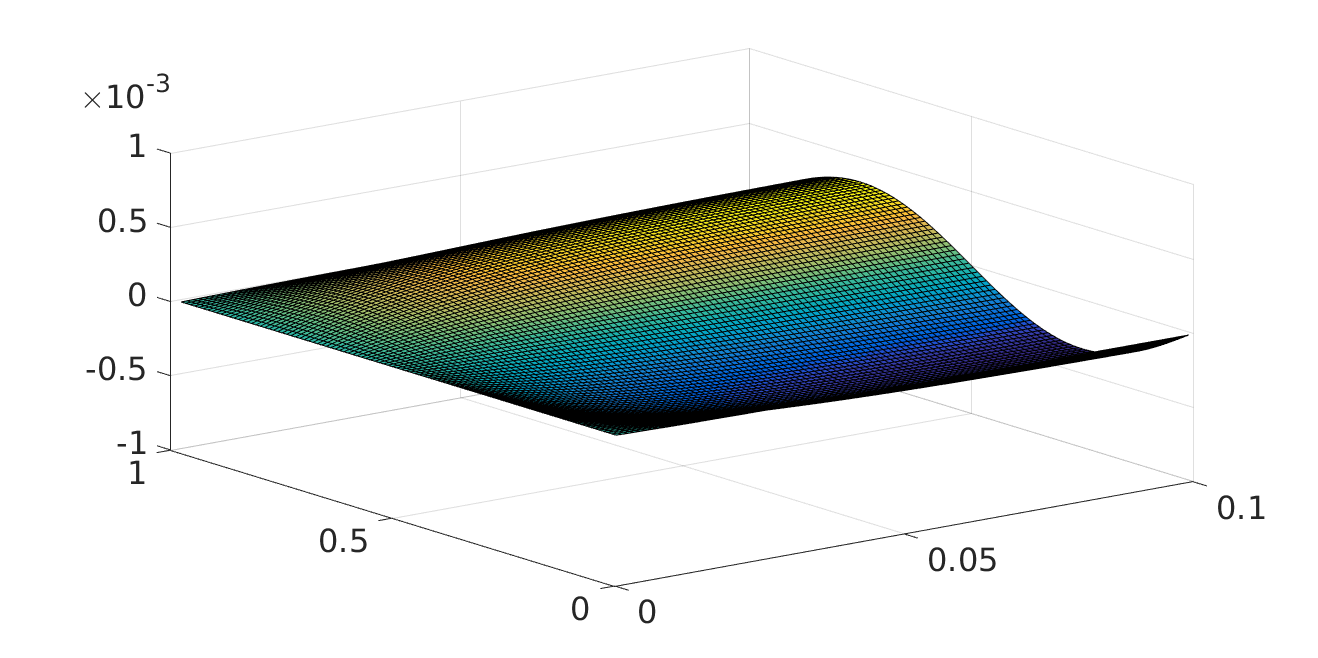}
\caption{Results of reduced Landweber method after 50000 steps:
Left: exact parameter $\thdag$ (solid) and reconstruction (dashed);
Middle: reconstructed state $\bu$;
Right: difference $\bu-\budag$ between exact and reconstructed state.
\label{fig:redLW}}
\end{figure}
\begin{figure}
\includegraphics[width=0.33\textwidth]{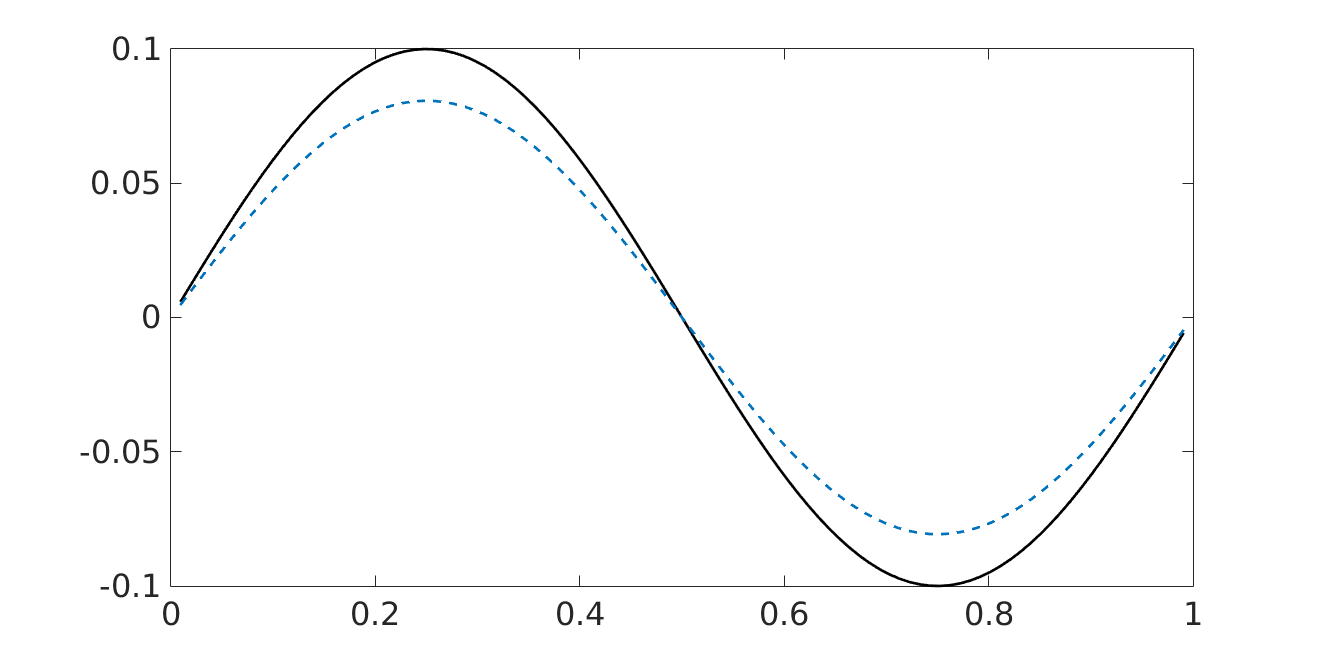}
\includegraphics[width=0.33\textwidth]{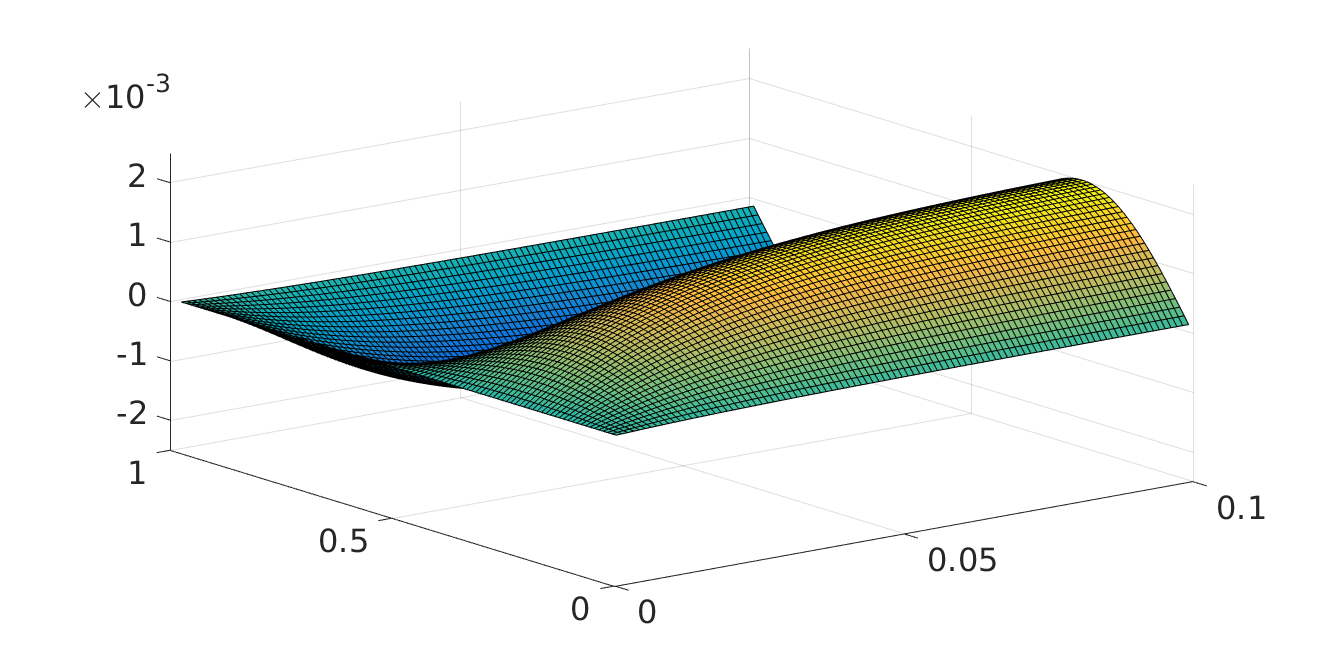}
\includegraphics[width=0.33\textwidth]{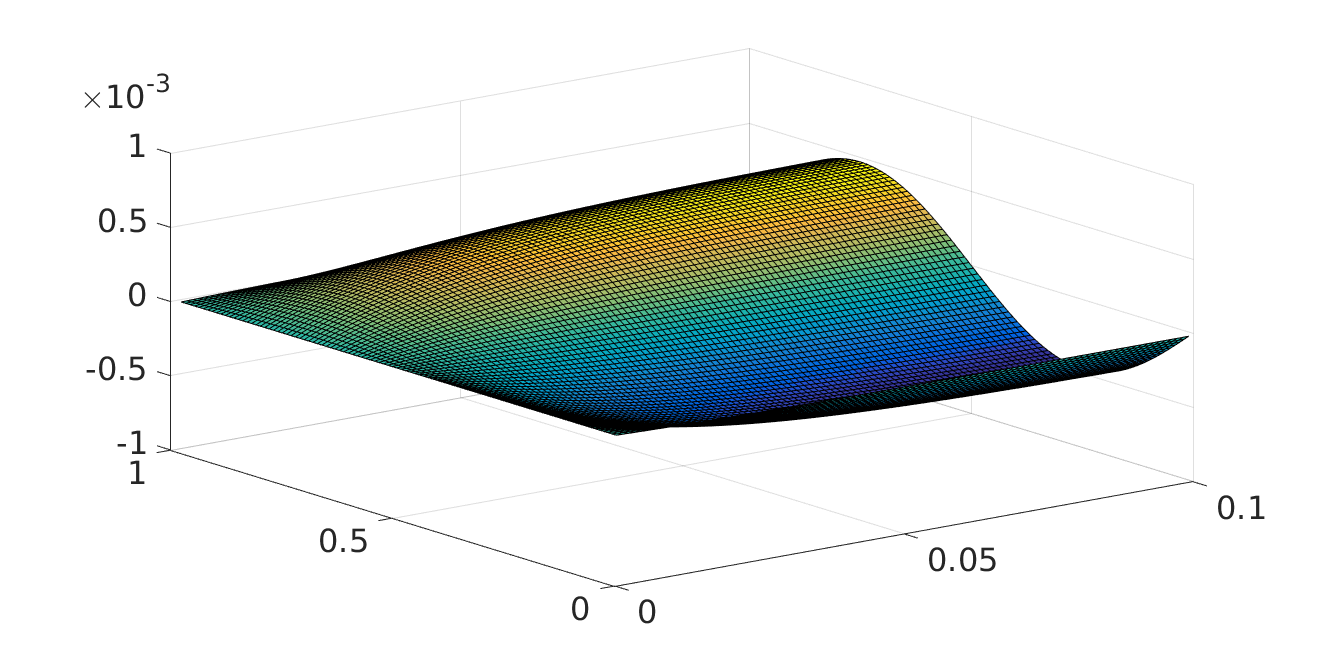}
\caption{Results of all-at-once Landweber method after 50000 steps:
Left: exact parameter $\thdag$ (solid) and reconstruction (dashed);
Middle: reconstructed state $\bu$;
Right: difference $\bu-\budag$ between exact and reconstructed state.
\label{fig:aaoLW}}
\end{figure}
More detailed tests for all methods under consideration, including perturbations of the state equation and the observations, will be subject of future work.   

\section{Conclusions and outlook} \label{sec:Concl}

In this paper we have started an investigation and comparison of some iterative regularization methods in all-at-once versus reduced formulations of time dependent inverse problems over finite time horizons.
It turned out that the resulting methods differ considerably between their all-at-once and their reduced versions. 
Also, we have derived some convergence results from the existing Hilbert space regularization theory under appropriate conditions adopted to this time dependent setting.  

\medskip

In this context, several further questions arise that are to be tackled in future research:

As mentioned in Remark \ref{rem:timedependenttheta} the theory here in principle allows for time dependent parameters $\theta$. However, details such as growth conditions on $\tilde{f}$, $\tilde{g}$ for proving differentiability, are yet to be investigated. An advantage of time dependent parameters is the possibility of verifying range invariance and hence convergence of the IRGNM type methods mentioned here, as well as the IRGNM Kaczmarz method from \cite{BuKa04}. 

Regularization theory in general Banach space would allow to use the weak form of the model equation according to Remark \ref{rem:weakformulation}, which might on one hand simplify the adjoints in $\cU$ and on the other hand enable stronger nonlinearities in the semilinear equation from Section \ref{sec:ex}, cf., e.g., \cite[Theorem 5.5]{TroeltzschBuch}. However, note that this would require to work with nonreflexive spaces such as $C(0,T;H)$, for which regularization theory for iterative methods is developed to a lesser extent.

While convergence without rates has been shown here, convergence rates require the exact solution to satisfy certain regularity conditions (source conditions) whose appearance in the time dependent setting is yet to be investigated. 

Also numerical implementation as well as computational tests and comparisons are subject of future work.

\section*{Acknowledgment}

The author gratefully acknowledges financial support by the Austrian Science Fund FWF under the grants I2271 ``Regularization and Discretization of Inverse Problems for PDEs in Banach Spaces''  and P30054 ``Solving Inverse Problems without Forward Operators'' as well as partial support by the Karl Popper Kolleg ``Modeling-Simulation-Optimization'', funded by the Alpen-Adria-Universit\" at Klagenfurt and by the Carinthian Economic Promotion Fund (KWF).

Moreover, we wish to thank both reviewers for fruitful comments leading to an improved version of the manuscript.

\end{document}